\documentclass[a4paper, 10pt, twoside]{article}

\usepackage{amsmath, amscd, amsfonts, amssymb, amsthm, latexsym, url, color, todonotes, bm, framed, rotating, enumerate} %pdflscape}
\usepackage{graphicx}
\usepackage[left=1in, right=1in, top=1.2in, bottom=1in, includefoot, headheight=13.6pt]{geometry}
\usepackage{mathtools}
\input{xy}
\xyoption{all}
\usepackage{caption} %

\usepackage{tikz}
\usetikzlibrary{babel}
\usepackage{quiver}
\usepackage{thmtools}
\usepackage{thm-restate}

\usepackage{tablefootnote}
\usepackage[section]{placeins}

\usepackage{threeparttable}
\usepackage{diagbox}
\newcommand{\ssep}{\textsuperscript{,}}

\usepackage[colorlinks,breaklinks=true]{hyperref}
\usepackage[figure,table]{hypcap}
\hypersetup{
	bookmarksnumbered,
	pdfstartview={FitH},
	citecolor={black},%{blue},
	linkcolor={black},%{red},
	urlcolor={black},
	pdfpagemode={UseOutlines}
}
\makeatletter
\newcommand\org@hypertarget{}
\let\org@hypertarget\hypertarget
\renewcommand\hypertarget[2]{%
  \Hy@raisedlink{\org@hypertarget{#1}{}}#2%
} 
\makeatother 

\newtheorem{theorem}{Theorem}[section]
\newtheorem{lemma}[theorem]{Lemma}
\newtheorem{corollary}[theorem]{Corollary}
\newtheorem{proposition}[theorem]{Proposition}

\theoremstyle{definition}
\newtheorem{definition}[theorem]{Definition}

\newtheorem{remark}[theorem]{Remark}
\newtheorem{example}[theorem]{Example}

\newcommand{\xysquare}[8]{
\[\xymatrix{
#1 \ar@{#5}[r] \ar@{#6}[d] & #2 \ar@{#7}[d]\\
#3 \ar@{#8}[r] & #4
}\]
}

\newcommand{\bb}{\mathbb}

\newcommand{\comment}[1]{}

\renewcommand{\phi}{\varphi}

\newcommand{\roi}{\mathcal{O}}

\usepackage{calc} % if not already loaded, for width calculations
\usepackage{enumitem}

\newlist{caselist}{enumerate}{1}

\setlist[caselist]{%
  label=\textnormal{Case (\roman*)}, % prints Case (i), Case (ii), ...
  align=left,                        % aligns label and text nicely
  leftmargin=*,                      % no label sticking into the margin
  itemsep=0pt                        % optional: tighten vertical space
}

\renewcommand{\cal}{\mathcal}
\renewcommand{\hat}{\widehat}

\newcommand{\indlim}{\varinjlim}
\renewcommand{\tilde}{\widetilde}
\renewcommand{\Im}{\operatorname{Im}}

\DeclareMathOperator{\Char}{char}

\DeclareMathOperator{\Frac}{Frac}

\DeclareMathOperator{\Jac}{Jac}

\DeclareMathOperator{\Pic}{Pic}

\DeclareMathOperator{\Spec}{Spec}
\DeclareMathOperator{\Proj}{Proj}

\DeclareMathOperator{\Tor}{Tor}

\DeclareMathOperator{\essdim}{essdim}

\newcommand{\CH}{C\!H}

\newcommand{\rcoeq}[3]{\xymatrix{ #1\ar@/^3mm/[r]^f \ar@/_3mm/[r]_g & #2 \ar[l]_e\ar[r] & #3}}

\usepackage[bbgreekl]{mathbbol}
\DeclareSymbolFontAlphabet{\mathbbm}{bbold}

\usepackage{fancyhdr}

\usepackage{sectsty}
\sectionfont{\Large\sc\centering}
\chapterfont{\large\sc\centering}
\chaptertitlefont{\LARGE\centering}
\partfont{\centering}

\begin{document}
\itemsep0pt

\title{On the diagonal of %some low bidegree
low bidegree hypersurfaces}

\author{Elia Fiammengo and Morten Lüders}

%\classno{}
\date{}
%\extraline{}

\maketitle

\begin{abstract}
We study the existence of a decomposition of the diagonal for bidegree hypersurfaces in a product of projective spaces. Using a cycle theoretic degeneration technique due to Lange, Pavic and Schreieder, we develop an inductive procedure that allows one to raise the degree and dimension starting from the quadric surface bundle of Hassett, Pirutka and Tschinkel. Furthermore, we are able to raise the dimension without raising the degree in a special case, showing that a very general $(3,2)$ complete intersection in $\bb P^4\times \bb P^3$ does not admit a decomposition of the diagonal. As a corollary of these theorems, we show that in a certain range, bidegree hypersurfaces which were previously only known to be stably irrational over fields of characteristic zero by results of Moe, Nicaise and Ottem, are not retract rational over fields of characteristic different from two.
\end{abstract}

\tableofcontents

\section{Introduction}
In this article we study the existence of a decomposition of the diagonal for bidegree hypersurfaces in a product of projective spaces. Rational, stably rational and even retract rational varieties admit a decomposition of the diagonal; equivalently, they have torsion order equal to one.
Since such a decomposition behaves well under specialisation, its existence has, following its introduction by Voisin \cite{VoisinInventiones}, been an important tool in disproving stable and retract rationality. Prominently, this tool has been used by Hassett--Pirutka--Tschinkel to show that the bidegree $(2,2)$ hypersurface  
 $$Q:=\{y_1y_2x_0^2+y_0y_2x_1^2+y_0y_1x_2^2+(y_0^2+y_1^2+y_2^2-2(y_0y_1+y_0y_2+y_1y_2))x_3^2=0\}$$
 $$\subset \Proj k[y_0,y_1,y_2]\times \Proj k[x_0,x_1,x_2,x_3]$$
 \cite[Example 8]{HassetPirutkaTschinkel2018} is not retract rational. Their argument exploits the quadric surface bundle structure on $Q$ arising from projection to the first factor and implies that $\Tor(Q)=2$ (see also Remark \ref{rem:HPT_unirational}).  Indeed, a bidegree $(d,f)$ hypersurface $X\subset \bb P^{n-r}\times \bb P^{r+1}$ can also be viewed as a family of degree $f$ hypersurfaces in $\bb P^{r+1}$ parametrised by $\bb P^{n-r}$. If $f\leq r+1$, then the $k$-variety $X$ is rationally connected \cite[IV.6.5]{Kollar1996}, and hence admits a decomposition of the diagonal with rational coefficients and its torsion order $\Tor (X)$ is finite. It is an interesting problem to investigate the number $\Tor(X)$, and substantial progress has been made in computing this invariant for many classes of rationally connected varieties (see, e.g.,\cite{ChatzistamatiouLevine},\cite{LangeSchreieder2024},\cite{LangeZhang}).
 
 As pointed out by Koll\'ar in \cite{Kollar2000}, bidegree hypersurfaces are particularly interesting when the fibers %of the projection to the first factor $X\to \bb P^{n-r}$ 
 are rational $r$-folds. For instance, $f=2$ yields quadric bundles (and conic bundles when $r=1$) and for $r=2$, %with
$f\in \{2,3\}$ one obtains del Pezzo fibrations.
%Other interesting del Pezzo fibrations are obtained by considering the cyclic covers of $\bb P^{n-2}\times \bb P^3$ ramified over $X$ (with $e\in \{2,3,4\}$). %Koll\'ar shows that
In these situations it has been shown that for a very general hypersurface in $\bb P^{n-1}_{\bb C}\times \bb P^2_{\bb C}$ of bidegree $(d,2)$ one has $\Tor (X)>1$ provided $d\geq n-1\geq 2$ \cite[Cor. 1.2]{AhmadinezhadOkada2018}; and for a very general hypersurface $X\subset \bb P_{\bb C}^{n-1}\times \bb P_{\bb C}^3$ of bidegree $(d,3)$ one has $\Tor (X)>1$ provided $d\geq n\geq 3$ \cite[Thm. 1.2(3)]{OkadaKrylov2020}.

The conic bundle case has been upgraded by Nicaise--Ottem \cite{NicaiseOttem2022} to show stable irrationality for $d\geq n-2$. Furthermore, they show that if $n\geq 5$, then very general hypersurfaces in $\bb P_{\bb C}^{n-2}\times \bb P_{\bb C}^3$ of bidegree $(d,2)$ are not stably rational if $d\geq n-3$. Their main tool is a theorem which allows one to raise the degree and dimension simultaneously in one component of the bidegree hypersurface $\{Q=0\}$. More precisely, they show that if $k$ is a field of characteristic zero and if a very general bidegree $(d,f)$ hypersurface in  $\bb P^{\ell}_k\times \bb P^m_k $ is stably irrational, then, for all $\ell'\geq \ell$ and $m'\geq m$ with
%\begin{equation*}
    $d'\geq d+\ell'-\ell\,\,\,\text{and}\,\,\,f'\geq f+m'-m$
%\end{equation*}
a very general bidegree $(d',f')$ hypersurface in $\bb P^{\ell'}_k\times \bb P ^{m'}_k$ is stably irrational (see \cite[Theorem 4.5]{NicaiseOttem2022}).

Most of the results of \cite{NicaiseOttem2022} have since been strengthened to obstruct a decomposition of the diagonal; see e.g. \cite{FiammengoLuedersCubics}, \cite{LangeSchreieder2024}, \cite{LangeSkauli2023}, \cite{LangeZhang}, \cite{PavicSchreieder2023}, \cite{Skauli2023}. Here we first prove an analogue of Theorem \cite[Theorem 4.5]{NicaiseOttem2022} for the decomposition of the diagonal, valid in arbitrary characteristic $p\neq2$:
\begin{restatable}{theorem}{theoremraisedegdim}\label{theorem_raise_deg_dim}
    Let $k$ be a field of characteristic different from two and $\Lambda$ a ring of positive characteristic such that the exponential characterristic of $k$ is invertible in $\Lambda$. Let $n\geq r+1\geq 1$ and assume there exists an integral bidegree $(d,f)$ hypersurface 
    $$X\subset \bb P^{n-r}_k\times \bb P^{r+1}_k=\Proj k[y_0,...,y_{n-r}]\times \Proj k[x_0,...,x_{r+1}]$$
    such that $\Tor^{\Lambda}(X,\{y_{0}x_{0}y_{1}x_{1}=0\})\neq 1$.
     Then for any integers $n'\geq n$, $f'\geq f$ and $d'\geq d+(n'-n)$ a very general bidegree $(d',f')$ hypersurface $X'$ in $\bb P^{n'-r}_k\times \bb P^{r+1}_k$ we have $\Tor^{\Lambda}(X',\{y_{0}x_{0}y_{1}x_{1}=0\})\neq 1$. In particular, $X'$ does not admit a decomposition of the diagonal.
\end{restatable}
In order to get good results from Theorem \ref{theorem_raise_deg_dim}, one needs examples of low bidegree hypersurfaces to start with. As in the work of Nicaise--Ottem, the %canonical 
examples to start with are the quadric surface bundle of Hassett--Pirutka--Tschinkel and the very general bidegree $(2,2)$ hypersurface in $\bb P^2\times \bb P^2$ (see Proposition \ref{Proposition:Tor_simple_subschemes}).
%\cite[8.2]{HassetTschinkel2019}. 
Another important example is the following: using his double cone construction and toric degeneration arguments, Moe has been able to raise the degree of $Q$ by one and the dimension by two and to show that a $(3,2)$ hypersurface in $\bb P^4_k\times \bb P^{3}_k$, where $k$ is a field of characteristic zero, is not stably rational \cite[Thm. 5.1]{Moe2023}. Inspired by this result we show the following theorem:

\begin{restatable}{theorem}{ThreeTwoTheorem}\label{theorem_3_2_in_4_3}
    Let $k$ be a field of characteristic different from two and $\Lambda:=\bb Z/2$. Let $X$ be a very general hypersurface of bidegree $(3,2)$ in $\bb P^4_k\times \bb P^3_k$. Then $\Tor^{\Lambda}(X,\{y_{0}x_{0}y_{1}x_{1}y_{2}x_2=0\})\neq 1$. In particular, $X$ does not admit a decomposition of the diagonal. 
\end{restatable}

As a consequence of the previous results,  
%Theorems \ref{theorem_raise_deg_dim} and \ref{theorem_3_2_in_4_3} and the quadric surface bundle of Hassett--Pirutka--Tschinkel, 
we obtain the following:

\begin{restatable}{theorem}{Maintheorem}\label{Theorem_bidegree_bounds}
   Let $k$ be a field of characteristic different from two and $n\in \bb N$, then a very general bidegree $(d,f)$ hypersurface $X$ in $\bb P^{n-r}_k\times \bb P^{r+1}_k$ does not admit a decomposition of the diagonal in the following cases:
   \begin{enumerate}
       \item $n\geq 4$, $r=1$, $f=2$ and $d\geq n-2$;
       \item $n= 5$, $r=2$, $f=2$, $d\geq 3$ and $n\geq 6$, $r=2$, $f=2$, $d\geq n-3$;
       
       \item $n\geq 5$, $r=2$, $f=3$ and $d\geq n-3$.
   \end{enumerate}
\end{restatable}

We mention two more corollaries. Theorem \ref{theorem_raise_deg_dim}, together with \cite[Theorem 1.4]{FiammengoLuedersCubics}, yields the following conditional statement:
\begin{restatable}{corollary}{conditional_reduction}\label{Cor:reduction_bideg_to_cubic}
    Let $k=\bb C$, $r\geq 3$ odd (or $r=4$). If a very general cubic hypersurface in $\bb P_k^{r+1}$ does not admit a decomposition of the diagonal, then for $d\geq n-r+1\geq 2$ a very general hypersurface of bidegree $(d,3)$ in $\bb P^{n-r}_k\times \bb P^{r+1}_k$ does not admit a decomposition of the diagonal.
\end{restatable}
Note that for $r=3, n=6$, Corollary \ref{Cor:reduction_bideg_to_cubic} is implied by Theorem \ref{theorem_3_2_in_4_3}, not however for $r>3$. The next corollary concerns Fano fibrations over $\bb P^1$.

\begin{restatable}{corollary}{TwoFourTheorem}\label{cor_2_4_in_1-2_6-5}
Let $k$ be a field of characteristic different from two and $\Lambda:=\bb Z/2$. Let $X$ be a very general hypersurface of bidegree $(2,4)$ in $\bb P^1_k\times \bb P^6_k$. Then $\Tor^{\Lambda}(X,\{y_{0}x_{0}y_{1}x_{1}=0\})\neq 1$. In particular, $X$ does not admit a decomposition of the diagonal. 
\end{restatable}
Theorems \ref{theorem_raise_deg_dim} and \ref{theorem_3_2_in_4_3} are proved in Sections \ref{Section:maintheorem} and \ref{Sect:Raising_dimension_without_raising_degree}, the heart of the article. The proofs rely heavily on a cycle theoretic method developed by Lange--Schreieder in \cite{LangeSchreieder2024} which allows to find obstructions to the existence of a decomposition of the diagonal in lower-dimensional strata of a strictly semistable degeneration of a variety. We recall these techniques in Section \ref{Section_Küchle_varieties} sketching a proof of the fact that a Küchle variety of type $b_4$ does not admit a decomposition of the diagonal.

The strategy of our argument to prove Theorem \ref{theorem_raise_deg_dim} can be summarised as follows: firstly, by a degeneration argument, the nonexistence of a decomposition of the diagonal of a hypersurface of bidegree $(d,f)$ in $\bb P^{n-r}\times \bb P^{r+1}$ implies the nonexistence of a decomposition of the diagonal for a $(d,f)-(2,0)$-complete intersection of dimension $n+1$ in $\bb P^{n-r+2}\times \bb P^{r+1}$. Secondly, this complete intersection can be shown to be birational to a hypersurface of bidegree $(d+1,f)$ in $\bb P^{n-r+1}\times \bb P^{r+1}$:
\begin{equation*}\label{equation_argument1}
(d,f) \text{ in }\bb P^{n-r}\times \bb P^{r+1} \rightsquigarrow (d,f)-(2,0) \text{-CI in }\bb P^{n-r+2}\times \bb P^{r+1} \rightsquigarrow (d+1,f) \text{ in }\bb P^{n-r+1}\times \bb P^{r+1}.    
\end{equation*}
This goes under the slogan ``raising the degree and the dimension''. For Theorem \ref{theorem_3_2_in_4_3} we apply this strategy first to $\{Q=0\}$ and then show that for the resulting hypersurface of bidegree $(3,2)$ in $\bb P^{3}\times \bb P^{3}$ we can just ``raise the dimension'', i.e. again we get the nonexistence of a decomposition of the diagonal for a $(3,2)-(2,0)$-complete intersection in $\bb P^{5}\times \bb P^{3}$ but this time this complete intersection can be shown to be birational to a hypersurface of bidegree $(3,2)$ in $\bb P^{4}\times \bb P^{3}$:
\begin{equation*}\label{equation_argument2}
(3,2) \text{ in }\bb P^{3}\times \bb P^{3} \rightsquigarrow (3,2)-(2,0) \text{-CI in }\bb P^{5}\times \bb P^{3} \rightsquigarrow (3,2) \text{ in }\bb P^{4}\times \bb P^{3}.    
\end{equation*}

\begin{remark} Theorem \ref{Theorem_bidegree_bounds}$(i)$ may be interpreted to concern conic bundles, $(ii)$ quadric surface bundles and $(iii)$ cubic surface bundles. At the same time it should be noted that there are many more such bundles and that there is the following optimal result on the rationality problem for quadric bundles over an $m$-dimensional rational base: 
\begin{theorem} \cite[Thm. 1]{Schreieder2019quadricbundles} 
    Let $n,r$ be positive integers with $r\leq 2^n-2$ and let $m\leq n$ be the unique integer with $2^{m-1}-1\leq r\leq 2^m-2$. Then there are smooth unirational complex $r$-fold quadric bundles $X$ over $S=\bb P^{n-m}\times\bb P^m$ such that $X$ does not admit a decomposition of the diagonal.
\end{theorem}
\noindent Note that the $(3,2)$ in $\bb P^4\times \bb P^3$ is 'covered' by that theorem in the following sense: for $r=m=2$ and $n=4$, there exists a quadric surface bundle over a rational $4$-fold which %is not stably rational/no dod.
does not admit a decomposition of the diagonal.
Schreieder proves this theorem by constructing Colliot-Th\'el\`ene--Ojanguren type quadric bundles. The bounds obtained for bidegree hypersurfaces are slightly weaker (i.e. $e\geq n+3$ for $r=2$, see \cite[Introd.]{Schreieder2019quadricbundles}).
\end{remark}

\begin{remark}
   There is an essential difference in the philosophy of Theorem \cite[Theorem 4.5 \text{ and } Prop. 6.1]{NicaiseOttem2022} and the philosophy of Theorem \cite[Prop. 6.2]{NicaiseOttem2022}. In the first, for a given bidegree hypersurface the degree and dimension in one factor are raised. In the second, a further projective space and further degree are added in the product, respectively multidegree. In \cite{LangeZhang}, with the goal of disproving the existence of a decomposition of the diagonal for multidegree hypersurfaces, Lange--Zhang pursue the second philosophy starting with Schreieder's singular hypersurfaces (of degree $\geq 4$) with nontrivial unramified cohomology classes. This is reflected in their key Theorem 3.12(b).
   Our approach is the first, i.e. to raise the dimension of the factors in order to get better results in low degree starting with quadric surface bundles. 
\end{remark}

In Section \ref{Section_foufolds_bidegree} we compute the relative torsion order of fourfolds of bidegree $(2,3)$ in $\bb P^1\times \bb P^4$ and thereby complete the classification of the very general bidegree threefolds and fourfolds with respect to geometric retract rationality. We finish the introduction by illustrating in tables what is now known about the torsion order of bidegree hypersurfaces in dimensions $3-6$. 
\begingroup
\renewcommand\thefootnote{}% keep as you like; no footnotes are used here

\begin{table}[ht]
\centering
\renewcommand{\arraystretch}{1.12}
\begin{tabular}{c|c|c|c|c}
 bidegree of $X$
& $\bb P^1\times \bb P^3$
& $\bb P^2\times \bb P^2$
& $\bb P^2\times \bb P^3$
& $\bb P^1\times \bb P^4$
\\ \hline
$(1,f)$ & R & R & R & R \\
$(d,1)$ & R & R & R & R \\
$(2,2)$
& R$^*$
& $\Tor(X)>1$\hyperlink{tab1refs}{\textsuperscript{1}}
& $\Tor(X)>1$\hyperlink{tab1refs}{\textsuperscript{3}}
& R$^*$ \\
$(2,f\ge 3)$
& $\Tor(X)>1$\hyperlink{tab1refs}{\textsuperscript{2}}
& $\Tor(X)>1$\hyperlink{tab1refs}{\textsuperscript{5}}
& $\Tor(X)>1$\hyperlink{tab1refs}{\textsuperscript{8}}\ssep\hyperlink{tab1refs}{\textsuperscript{9}}
& $\Tor(X)>1$\hyperlink{tab1refs}{\textsuperscript{4}} \\
$(d\ge 3,2)$
& R$^*$
& $\Tor(X)>1$\hyperlink{tab1refs}{\textsuperscript{5}}
& $\Tor(X)>1$\hyperlink{tab1refs}{\textsuperscript{7}}
& R$^*$ \\
$(d\ge 3,f\geq 3)$
& $\Tor(X)>1$\hyperlink{tab1refs}{\textsuperscript{2}}
& $\Tor(X)>1$\hyperlink{tab1refs}{\textsuperscript{5}}\ssep\hyperlink{tab1refs}{\textsuperscript{10}}
& $\Tor(X)>1$\hyperlink{tab1refs}{\textsuperscript{7}}
& $\Tor(X)>1$\hyperlink{tab1refs}{\textsuperscript{6}} \\
\end{tabular}

\caption{Dimensions $3$ and $4$: $\Tor(X)>1$ indicates failure of a decomposition of the diagonal, R indicates (geometric) rationality and R$^*$ indicates (geometric) rationality by a theorem of Lang (see Remark \ref{Rem:Geo-nonGeo}).\;
\protect\hypertarget{tab1refs}{}%
\textnormal{(References: }\textsuperscript{1}\cite{HassetTschinkel2019},\;
\textsuperscript{2}\cite{OkadaKrylov2020},\;
\textsuperscript{3}\cite{HassetPirutkaTschinkel2018},\;
\textsuperscript{4} Prop.~\ref{prop:proposition_(2,3)_in_P^1timesP^4},\;
\textsuperscript{5}\cite{HassetKreshTschinkel},\;
\textsuperscript{6} Thm.~\ref{theorem_raise_deg_dim},\;
\textsuperscript{7}\cite{Schreieder2018_quadricbundles},\;
\textsuperscript{8}\cite{AuelBohningPirutka2018},\;
\textsuperscript{9}\cite{AhmadinezhadOkada2018},\;
\textsuperscript{10}\protect\cite{Pirutka2025}\textnormal{).}}
\label{Table:1}
\end{table}

\endgroup
%The following incomplete table illustrates the above Theorems in dimension five and six.
\FloatBarrier
\begingroup
\renewcommand\thefootnote{}% keep as you like; no footnotes are used here

\begin{table}[ht]
\centering
\renewcommand{\arraystretch}{1.12}
\begin{tabular}{c|c|c|c|c|c|c}
bidegree of $X$
& $\bb P^1\times \bb P^5$
& $\bb P^2\times \bb P^4$
& $\bb P^3\times \bb P^3$
& $\bb P^1\times \bb P^6$
& $\bb P^2\times \bb P^5$
& $\bb P^3\times \bb P^4$
\\ \hline
$(1,f)$ & R & R & R & R & R & R \\
$(d,1)$ & R & R & R & R & R & R \\

$(2,2)$
& R$^*$ & R$^*$ & ? & R$^*$ & R$^*$ & ? \\

$(3,2)$
& R$^*$ & R$^*$ & $\Tor(X)>1\hyperlink{tab2refs}{\textsuperscript{1}}$ & R$^*$ & R$^*$ & ? \\

$(2,3)$
& ? & $\Tor(X)>1\hyperlink{tab2refs}{\textsuperscript{1}}$ & $\Tor(X)>1\hyperlink{tab2refs}{\textsuperscript{1}}$ & ? & ? & $\Tor(X)>1\hyperlink{tab2refs}{\textsuperscript{2}}$ \\

$(3,3)$
& ? & $\Tor(X)>1\hyperlink{tab2refs}{\textsuperscript{1}}$ & $\Tor(X)>1\hyperlink{tab2refs}{\textsuperscript{1}}$ & ? & ? & $\Tor(X)>1\hyperlink{tab2refs}{\textsuperscript{2}}$ \\

$(d\geq 2,f\geq 4)$
& $\Tor(X)>1\hyperlink{tab2refs}{\textsuperscript{4}}$ &
$\Tor(X)>1\hyperlink{tab2refs}{\textsuperscript{4}}\ssep\hyperlink{tab1refs}{\textsuperscript{1}}$ &
$\Tor(X)>1\hyperlink{tab2refs}{\textsuperscript{1}}$
& $\Tor(X)>1\hyperlink{tab2refs}{\textsuperscript{4}}\ssep\hyperlink{tab1refs}{\textsuperscript{3}}$ &
$\Tor(X)>1\hyperlink{tab2refs}{\textsuperscript{4}}\ssep\hyperlink{tab1refs}{\textsuperscript{3}}$ &
$\Tor(X)>1\hyperlink{tab2refs}{\textsuperscript{4}}\ssep\hyperlink{tab1refs}{\textsuperscript{2}}$ \\

$(d\geq 4,f= 2)$
& R$^*$ & R$^*$ & $\Tor(X)>1\hyperlink{tab2refs}{\textsuperscript{1}}$
& R$^*$ & R$^*$ & $\Tor(X)>1\hyperlink{tab2refs}{\textsuperscript{2}}$ \\

$(d\geq 4,f= 3)$
& ? & $\Tor(X)>1\hyperlink{tab2refs}{\textsuperscript{1}}$ & $\Tor(X)>1\hyperlink{tab2refs}{\textsuperscript{1}}$
& ? & ? & $\Tor(X)>1\hyperlink{tab2refs}{\textsuperscript{2}}$ \\
\end{tabular}

\caption{Dimensions $5$ and $6$: the notation is as in Table \ref{Table:1}; the cases marked with ? are open.\;
\protect\hypertarget{tab2refs}{}%
\textnormal{(References: }\textsuperscript{1} Thm.~\ref{Theorem_bidegree_bounds},\;
\textsuperscript{2} Thm.~\ref{theorem_3_2_in_4_3},\;
\textsuperscript{3} Cor.~\ref{cor_2_4_in_1-2_6-5},\;
\textsuperscript{4} \protect\cite[Theorem~5.7]{LangeZhang} (in the two columns on the right if $d,f$ are large enough)\textnormal{).}}
\label{Table:2}
\end{table}

\endgroup

\begin{remark}\label{Rem:Geo-nonGeo}
By Springer's theorem \cite{Springer}, if a flat morphism of projective varieties $f:X\to S$ is a quadric bundle and the base $S$ is rational, then $X$ is rational if $f$ admits a multisection of odd degree. By a theorem of Lang (see \cite[II. 4.5]{SerreGalCoh}), $f$ admits a section if $\dim(\text{fib}(f))>s^{\dim S}-2$ and if the ground field is algebraically closed.
    The cases marked by R$^*$ in Table \ref{Table:1} are therefore rational by Lang's Theorem. %as the base field is assumed algebraically closed. 
   % Over non-closed fields rationality of these quadric fibrations is a more subtle problem (see e.g. \cite{CTPS2025}).
\end{remark}

\begin{remark}
    It follows from Corollary \ref{cor_2_4_in_1-2_6-5} and Theorem \ref{theorem_raise_deg_dim} that the Conjecture \cite[Conj. 1.6]{AhmadinezhadOkada2018} holds for retract rationality of bidegree hypersurfaces of dimension at most seven.
\end{remark}

\paragraph{Acknowledgements.} The second author is funded by the Deutsche Forschungsgemeinschaft (DFG, German Research Foundation), project number $557768455$ and by the Deutsche Forschungsgemeinschaft (DFG, German Research Foundation)
TRR 326 \textit{Geometry and Arithmetic of Uniformized Structures}, project number 444845124.

\section{Obstruction to the relative decomposition of the diagonal}
%\begin{proposition}\label{Proposition_unramified_relativetor}
    %Let $k$ be an algebraically closed field, $m\geq 2$ and $Z$ an integral proper $k$-variety. Let $W\subset Z$ be a closed subscheme with $Z-W$ smooth over $k$. Suppose there is an unramified class 
    %$$0\neq \gamma \in H_{nr}^n(k(Z),\mu_m^{\otimes n})$$ 
    %of order $m$ such that there exists an alteration $\tau:Z'\to Z$ of degree coprime to $m$ and for all subvarieties $E\subset Z'$ with $\tau(E)\subset W$, the restriction vanishes $\tau^*\gamma|_{k(E)}=0$. Then 
    %\begin{equation*}
     %   \Tor^\Lambda(Z,W)=m.
    %\end{equation*}
%    \end{proposition}

\subsection{The relative torsion order}
We recall the definition of the relative torsion order and some of its properties from \cite{LangeSchreieder2024}. For more details we refer to loc. cit.

 %By the localisation sequence there exists some zero-dimensional closed subset $W\subset X$ relative to which it admits a decomposition of the diagonal iff 

\begin{definition}
Let $X$ be a variety over a field $k$ and let $\Lambda$ be a ring. 
\begin{enumerate}
    \item $X$ admits a \textit{decomposition of the diagonal} if there exists a zero-cycle $z\in \CH_0(X)$ and a cycle $Z\in Z_{\dim X}(X\times_k X)$ whose support does not dominate the second factor of $X\times_k X$ such that
 $$e\cdot\Delta_X=z\times X+Z\in \CH_{\dim X}(X\times_k X)$$
 for $e=1$. The smallest integer $e$ such that a decomposition as above exists is also called the \textit{torsion order} of $X$ and denoted by $\Tor(X)$.
\item Let $\delta_X$ denote the image of the diagonal $\Delta\subset X\times X$ in $\CH_0(X_{k(X)},\Lambda)=\indlim_{V\subset X}\CH_{\dim X}(X\times V,\Lambda)$. Here $\CH_0(X_{k(X)},\Lambda):=\CH_0(X_{k(X)})\otimes_{\bb Z}\Lambda$. We say that $X$ \textit{admits a $\Lambda$-decomposition of the diagonal relative to a closed subset $W\subset X$} if
    $$\delta_X\in \Im (\CH_0(W_{k(X)},\Lambda)\to \CH_0(X_{k(X)},\Lambda)).$$
\item    The $\Lambda$-\textit{torsion order} of $X$ relative to a closed subset $W\subset X$, denoted by $\Tor^\Lambda(X,W)$, is the order of the element 
    $$\delta_X|_U=\delta_U\in \CH_0(U_{k(X)},\Lambda).$$
    \end{enumerate}
\end{definition}

\begin{remark}\label{Remark_LS_3.4}
    Note that $\Tor^\Lambda(X,W)= \Tor^\Lambda(U,\emptyset).$
Furthermore, $\Tor^\Lambda(X,W)=1$ iff $X$ admits a $\Lambda$-decomposition of the diagonal relative to $W$ \cite[Rem. 3.4]{LangeSchreieder2024}.
\end{remark}

The torsion order has the following important properties which we will need:
\begin{lemma}\label{Lemma_LS_3.6} \cite[Lem. 3.6]{LangeSchreieder2024}
    Let $X$ be a variety over a field $k$ and let $W\subset X$ be closed. Then the following hold:
    \begin{enumerate}
        \item[(a)] For all $m\in \bb Z$, $\Tor^{\bb Z/m}(X,W)|\Tor^{\bb Z}(X,W)$.
        \item[(b)] Let $W'\subset W\subset X$ be closed subsets, then $\Tor^{\Lambda}(X,W)|\Tor^{\Lambda}(X,W')$. 
        \item[(c)] $\Tor(X)$ is the minimum of the relative torsion orders $\Tor^{\bb Z}(X,W)$ where $W\subset X$ runs through all closed subsets of dimension zero.
         \item[(d)] If $\deg :\CH_0(X)\to \bb Z$ is an isomorphism, then $\Tor(X)=\Tor^\bb Z(X,W)$ for any closed subset $W\subset X$ of dimension zero which contains a zero-cycle of degree $1$.
        \item[(e)] If $k=\bar k$ is algebraically closed, then $\Tor^{\Lambda}(X,W)=\Tor^{\Lambda}(X_L,W_L)$ for any ring $\Lambda$ and any field extension $L/k$.
    \end{enumerate}
\end{lemma}
\subsection{Torsion order relative to $CH_0$-bounded and universally $CH_0$-trivial varieties}
Let $k$ be a field and $X$ a $k$-variety, recall that a zero cycle $\sum_pn_pp\in Z_0(X)$ is effective if $n_p\geq 0$ for all closed points $p\in X$. The following notion appears in \cite{Voisin2025delPezzo} and \cite{Voisi2025Fanos} and generalizes universal $CH_0$-triviality.
\begin{definition}\cite[Def. 1.6]{Voisi2025Fanos}
  Let $X$ be a proper $k$-scheme. We say that $X$ has bounded $CH_0$-group if there exists a number $N$
(possibly depending on $X$ and $k$), such that for any field $L/k$, any $0$-cycle $z\in CH_0(X_L)$ of degree deg $z\geq N$ is effective (that is, rationally equivalent to an effective $0$-cycle). $X$ has unbounded $CH_0$-group if such an integer $N$ does not exist.
\end{definition}

\begin{example}
\begin{enumerate}
    \item 
    Let $k$ be algebraically closed, then every smooth proper curve $X$ over $k$ is bounded by Riemann-Roch (take $N$ as the genus of $X$).
    \item For any field $k$, a universally $CH_0$-trivial $k$-variety has bounded $CH_0$-group (with $N=0$). In particular, by \cite[Corollary 2.2]{Voisin_Cubics2017} for $X$ a smooth complex projective surface with $\CH_0(X)\cong \bb Z$ and $H^*(X,\bb Z)_{\mathrm{tor}}=0$, $X$ has bounded $CH_0$-group. 
    \end{enumerate}
\end{example}
\begin{remark}
    Universal $CH_0$-triviality and $CH_0$-boundedness are easy to distinguish in dimension one (complex elliptic curves have huge yet bounded $CH_0$). The main tools to detect these notions are sections in $H^0(X,\wedge^i\Omega_X)$, which vanish in the range $i\geq 1$ for universal $CH_0$-triviality and in the range $i\geq 2$ for $CH_0$-boundedness. In the case of an elliptic curve $E$, $H^0(E,\Omega_E=\roi_E)=k$ and all higher differentials vanish for dimension reasons.
\end{remark}
\begin{lemma}\label{Lemma:Tor_simple_subschemes}
    Let $k$ be a field of characteristic zero, $X$ be a smooth proper $k$-variety and $f:W\hookrightarrow X$ a closed $k$-subscheme. Assume at least one of the following holds: 
    \begin{enumerate}
        \item $X$ is not universally $\CH_0$-trivial and $W$ is universally $\CH_0$-trivial.
        \item $X$ is not $\CH_0$-bounded and $W$ is $\CH_0$-bounded.
    \end{enumerate}
    Then $\Tor^\bb Z(X,W)>1.$
\end{lemma}
\begin{proof}
    Assume by contradiction that $\Tor^\bb Z(X,W)=1$. It then follows from \cite[Lem. 3.7]{LangeSchreieder2024} that the pushforward $f_*:CH_0(W_L)\to CH_0(X_L)$ is surjective for all extensions $L/k$. As the degree is compatible with proper pushforward, i.e. $\deg f_*[z]=\deg [z]$ for $[z]\in \CH_0(W_L)$ and the diagram
$$\xymatrix{
CH_0(W_L) \ar[dr]_{\deg} \ar[rr]^{f_*} & & CH_0(X_L) \ar[dl]^{\deg}\\
 & \bb Z          &}$$   
    commutes, this implies that neither $(i)$ nor $(ii)$ can hold.
\end{proof}
\begin{proposition}\label{Proposition:Tor_simple_subschemes}
    Let $k$ be an algebraically closed field of characteristic zero, $X$ be a very general bidegree $(2,2)$ hypersurface in $\Proj k[y_0,y_1,y_2]\times \Proj k[x_0,x_1,x_2]$, then $\infty>\Tor^\bb Z(X,\{\prod x_i\prod y_i=0\})>1$.
\end{proposition}
\begin{proof}
    By \cite[8.2]{HassetTschinkel2019} (see also \cite[Theorem 6.2]{AuelBigazziet.al.2021}) the $k$-variety $X$ is not universally $CH_0$-trivial.  By \cite[Lemma 2.4]{CTPirutka:cyclic_covers} and symmetry of the equations, in order to apply Lemma \ref{Lemma:Tor_simple_subschemes} it is enough to show that the divisors $S_i=\{x_i=0\}\subset X$ are universally $CH_0$-trivial. As $X$ is very general, Bertini's theorem implies that the $S_i$ are smooth integral bidegree $(2,2)$ hypersurfaces in $\bb P^1_k\times \bb P^2_k$. Denote by $H_1$ and $H_2$ the generators of $\Pic (\bb P^1_k\times \bb P^2_k)$. Since $K_{\bb P^1_k\times \bb P^2_k}=-2H_1-3H_2$, it follows by the adjunction formula that 
    $$K_{S_i}=(K_{\bb P^1_k\times \bb P^2_k}+S_i)|_{S_i}=(-2H_1-3H_2+2H_1+2H_2)_{S_i}=-H_2|_{S_i}$$
    and therefore $-K_{S_i}=H_2|_{S_i}$ is ample as the projection to the second factor $pr_2:S\to \bb P^2$ is finite for a very general choice of $X$ \cite[Tag 0892]{stacks-project}.    
     It follows that $S_i$ is a rational del Pezzo surface and is universally $\CH_0$-trivial. 

     Finally note that $\infty>\Tor^\bb Z(X,\{\prod x_i\prod y_i=0\})$ since $X$ is Fano with anticanonical divisor $\cal O_X(1,1)$ (see e.g. \cite[7.2]{Schreiedersurvey} and use Lemma \ref{Lemma_LS_3.6}(b)).
\end{proof}
\begin{remark}
    Many interesting threefolds that do not admit a decomposition of the diagonal have $\essdim =3$ by \cite{Mboro2019}, we do not know if a very general bidegree $(2,2)$ hypersurface $X$ in $\bb P^2\times \bb P^2$ also satisfies $\essdim X=3$.  
\end{remark}
\subsection{Unramified cohomology classes and relative torsion order}
In this section we extend \cite[Prop. 3.1]{SchreiederAMS} to a criterion which allows %to not just
one not only to disprove the existence of a decomposition of the diagonal but also the existence of a relative decomposition of the diagonal. This is well-known to the expert and sketched in \cite[Proposition 6.1]{Schreieder2021}. We state and give details for the convenience of the reader and for future reference.

\begin{proposition}\label{unramifiedobstruction_relativedecomposition} \cite[Prop. 3.1]{SchreiederAMS}\cite[Proposition 6.1]{Schreieder2021}
    Let $Y$ be an integral variety over an algebraically closed field $k$ and let $W\subset Y$ be a closed subscheme. 
    %containing $Y^{sing}$. 
    Let $S$ be a finite set of closed points of $Y$. Let $\ell$ be a prime different from $ch(k)$ and let $\tau:Y'\to Y$ be an alteration whose degree is prime to $\ell$. Suppose that for some $i\geq 1$ there is a nontrivial class (resp. class of order $\ell$) $\alpha\in H^i_{nr}(k(Y)/k,\bb Z/\ell)$ such that 
    $$(\tau^*\alpha)|_E=0\in H^i(k(E),\bb Z/\ell)$$
    for any subvariety $E\subset (\tau^{-1}(W)\cup \tau^{-1}(Y^{\mathrm{sing}}))$, then $Y$ admits neither an integral nor a $\Lambda=\bb Z/\ell$-decomposition of the diagonal relative to $$V:=W\cup S\cup Y^{\mathrm{sing}}$$ 
    (resp. $\ell\,|\,\Tor^{\bb Z/\ell}(X,V)$.).
\end{proposition}
\begin{remark}   It is possible to include $S$ in $W$ and to demand that $\tau^*\alpha$ vanishes along any subvariety of $\tau^{-1}(W)$, however, we will see in the course of the proof that this is always satisfied because of cohomological dimension. Stated this way the proposition is therefore easier to apply. 
\end{remark}

\begin{proof}
    We closely follow the proof of \cite[Prop. 3.1]{SchreiederAMS} and \cite[Thm. 6.1, last paragraph]{LangeSchreieder2024}.

    Let $K=k(Y)$. Assume that $Y$ admits a decomposition of the diagonal relative to $W\cup S\cup Y^{\mathrm{sing}}$. Let $U=Y\setminus (W\cup Y^{\mathrm{sing}})\subset Y$ and $U':=\tau^{-1}(U)$.  Then 
    \begin{equation}\label{equation1}
        \delta_Y|_{U_K}=[s_K]+\ell\cdot[\xi_K]\in \CH_0(U_K),
    \end{equation}
    where $\delta_Y$ denotes the class of the diagonal, $s_K$ is the base change of a zero cycle supported on $S$ and $\xi_K\in\CH_0(U_K)$ reflects that we work with $\bb Z/\ell$-coefficients. %I.o.w. $\delta_Y\in \mathrm{im}[\CH_0(W_K,\bb Z/\ell)\to \CH_0(Y_K,\bb Z/\ell)]$.

    Next we consider the commutative diagram
   $$\xymatrix{
U'_K \ar[d]_{\tau|_{U'}} \ar[r]^{j'} & Y_{K}'\ar[d]^\tau\\
U_K \ar[r]_j          & Y_K.}$$ 
%Applying $\tau|_{U'}^*\circ j^*$ to (\ref{equation1}), w
We get
    \begin{equation}
        \tau|_{U'}^*\circ j^*(\delta_Y)=[s_K]+\ell\cdot[\xi_K]\in \CH_0(U_K'),
    \end{equation}
    %\begin{equation}
%\begin{split}
%    \tau|_{U'}^*(j^*\delta_Y)= \tau|_{U'}^*(j^*[z_K])+\tau|_{U'}^*(j^*[s_K])+\ell\tau|_{U'}^*(j^*\xi)\in\CH_0(U_K').
%    \end{split}
%\end{equation}
 %   $z$ is a zero-cycle whose support lies in $W_{K}\cup (Y^{\mathrm{sing}})_{K}$,
where 
\begin{equation}
\begin{split}
%\tau|_{U'}^*(j^*[z_K])&= [z_K']\in\CH_0(U_K'), \\
\tau|_{U'}^*(j^*[s_K])&= [s_K']\in\CH_0(U_K'), \\
\;\ell\tau|_{U'}^*(j^*\xi)&=\ell [\xi']\in\CH_0(U_K'),
\end{split}
\end{equation}
with $s_K'$ the pullback (which exists for morphisms of smooth varieties by \cite[Sec. 8]{Fulton1998}) of the restriction $j^*[s_K]$ and $\xi'$ the pullback of the restriction $j^*[\xi]$.
Let $\Gamma_\tau\subset Y'\times Y$ be the graph of $\tau$ and $\delta_Y'\in \CH_0(Y_K')$ the zero-cycle given by the generic point of $\Gamma_\tau$. Since $U_K'\to U_K$ is \'etale in a neighborhood of the diagonal point, 
\begin{equation}
     \delta_Y'|_{U_K'}=\tau|_{U'}^*(j^*\delta_Y)\in\CH_0(U_K').
\end{equation}
Therefore the above equation becomes $\delta_Y'|_{U_K'}=[s_K]+\ell\cdot[\xi_K]\in \CH_0(Y_K)$.
By the localisation sequence
$$ \CH_0(Y_K'\setminus U_K')\to \CH_0(Y_K')\to \CH_0(U_K')\to 0$$
we get that
\begin{equation}\label{eq6}
    \delta_Y'=[s_K']+\ell[\xi']+[\tilde{z}]\in\CH_0(U_K'),
\end{equation}
where $\tilde{z}$ is a zero-cycle supported on $Y_K'\setminus U_K'$. 

Next we consider the Merkurjev pairing \cite{Merkurjev2008} (which works for smooth or snc schemes \cite[Sec. 5,6]{Schreiedersurvey}, which is the reason for working with alterations)
$$\CH_0(Y_K')\times H^i_{nr}(k(Y')/k,\bb Z/\ell)\to H^i(K=k(Y),\bb Z/\ell),([z],\beta)\mapsto \langle [z],\beta\rangle,$$
and which, if $z$ is represented by a closed point on $Y_K'$, satisfies $\langle [z],\beta\rangle=f_*(\beta_K|_z),$ where $\beta_K$ is the pullback along $Y'\times Y\to Y$, noticing that $k(Y'\times Y)=K(Y'_K)$, and $f:\Spec k(z)\to \Spec K$ a finite morphism. Indeed, there is a commutative diagram of pairings
%$$\xymatrix{
%\CH_0(Y') \ar[d] & \times  & H^i_{nr}(k(Y')/k,\bb Z/\ell) \ar[d]  \ar[r] & H^i(k,\bb Z/\ell) \ar[d]  \\
%\CH_0(Y_K') \ar@{=}[d] & \times  & H^i_{nr}(k(Y')/k,\bb Z/\ell) \ar[d]  \ar[r] & H^i(K,\bb Z/\ell) \ar@{=}[d]  \\
%\CH_0(Y_K')  & \times  & H^i_{nr}(K(Y')/K,\bb Z/\ell) \ar[r] & H^i(K,\bb Z/\ell)
%}
%$$
\begin{equation*}
\xymatrix@R=2.2em@C=0.45em{
\CH_0(Y') \ar[d] & \times &
H^i_{nr}\!\bigl(k(Y')/k,\bb Z/\ell\bigr) \ar[d] \ar[rr] &
\hspace{2.2em} &
H^i\!\bigl(k,\bb Z/\ell\bigr) \ar[d] \\
\CH_0(Y'_K) \ar[d]^{=} & \times &
H^i_{nr}\!\bigl(k(Y')/k,\bb Z/\ell\bigr) \ar[d] \ar[rr] &
\hspace{2.2em} &
H^i\!\bigl(K,\bb Z/\ell\bigr) \ar[d]^{=} \\
\CH_0(Y'_K) & \times &
H^i_{nr}\!\bigl(K(Y')/K,\bb Z/\ell\bigr) \ar[rr] &
\hspace{2.2em} &
H^i\!\bigl(K,\bb Z/\ell\bigr)
}
\end{equation*}
in which the bottom one in the Merkurjev pairing, inducing the middle one by the above procedure. The top pairing is again the original Merkurjev pairing; the middle one is the one we need.

We pair the class $\tau^*\alpha\in H^i_{nr}(k(Y')/k,\bb Z/\ell)$ with both sides of (\ref{eq6}). Let us first consider the lefthandside. Recall that $\Gamma_\tau\cong Y'$ and that its generic point $\delta_Y'$ has residue field $k(Y')$ and $\Spec k(Y')\to \Spec K$ is induced by $\tau$. Therefore by the definition of the pairing for closed points,
$$\langle\delta'_Y,\tau^*\alpha\rangle=\tau_*\tau^*\alpha=\deg(\tau)\alpha\neq 0\in H^i(K,\bb Z/\ell)$$
is non-zero because $(\deg \tau,\ell)=1$. Next, we consider the righthandside. First note that 
$$\langle \ell\xi',\tau^*\alpha\rangle=\langle \xi',\ell\tau^*\alpha\rangle=\langle \xi',\tau^*(\ell\alpha)\rangle=0$$
because $\alpha$ is $\ell$-torsion. Furthermore, $$\langle s_K',\tau^*\alpha\rangle=0$$
because $s'_K$ is the base change of a zero-cycle $s'$ on $Y'$ and therefore the pairing factors through the restriction of $\tau^*\alpha$ to $z'\in \CH_0(Y')$, i.e. the upper pairing in the above diagram, which vanishes because $H^i_{}(k,\bb Z/\ell)=0$ since $i\geq 1$ and $k$ is algebraically closed.
Finally,
$$\langle \tilde{z},\tau^*\alpha\rangle=0$$
because $\tau^*\alpha$ restricts to zero on the generic points of subvarieties of $\tau^{-1}(W)\cup \tau^{-1}(Y^{\mathrm{sing}})$ and hence pairs to zero with $\tilde{z}$ by assumption. Indeed, there exists a commutative diagram 
\begin{center}
\begin{tikzcd}
	{Z_{\dim Y}(Y\times Y')\times H^i_{nr}(k(Y')/k,\bb Z/\ell)} && {H^i(K,\bb Z/\ell)} \\
	& {Z_0(Y'_K)\times H^i_{nr}(K(Y')/K,\bb Z/\ell)}
	\arrow[from=1-1, to=1-3]
	\arrow[from=1-1, to=2-2]
	\arrow[from=2-2, to=1-3]
\end{tikzcd}
\end{center}
in which the upper horizontal map sends a tuple $(\Gamma,\beta)$ to $0$ if $\Gamma$ does not dominate the first factor and otherwise to $\Gamma^*\beta:=p_*((pr_{Y'}^*\beta)|_\gamma)$, where the first projection induces a finite morphism $p: \kappa(\gamma)\to K$ and $\gamma$ denotes the generic point of $\Gamma$ (see for example \cite[Cor. 5.3]{Schreiedersurvey}). Let us consider this for a closed point $y'\in Y_K'$. Let $y'':=pr_{Y'}(y')$.
Then $y''$ corresponds to a subvariety that under $\tau$ maps to $Y\setminus (W\cup Y^{\mathrm{sing}})$. Let $y:=\tau(y'')$. Now $p_*({(\tau^*\alpha)}|_{y'})=0$ 
since $((\tau^*\alpha)_K)|_{y'}=(\tau^*\alpha)|_{y''}$ as the diagram of pullbacks
$$\xymatrix{
H^i(k(y'),\bb Z/\ell) & H^i(\roi_{Y\times Y',y'},\bb Z/\ell)\ \ar[l]\\
H^i(k(y''),\bb Z/\ell) \ar[u]          & H^i(\roi_{Y',y''},\bb Z/\ell)\ar[u] \ar[l]} $$
commutes and $(\tau^*\alpha)|_{y''}=0$ by assumption.

In sum, we get that
$$\langle s'_K+\ell\xi'+\tilde z,\tau^*\alpha\rangle=0$$
which is a contradiction.
\end{proof}

\begin{example}\label{example_HPT}
    Let $k$ be an algebraically closed field of characteristic different from $2$ and $\Lambda=\bb Z/2$.
    Let
 $$\{Q=0\}\subset \Proj k[y_0,y_1,y_2]\times \Proj k[x_0,x_1,x_2,x_3].$$
 be the singular quadric surface bundle of \cite[Example 8]{HassetPirutkaTschinkel2018} defined by the bidegree $(2,2)$ homogeneous polynomial 
\begin{equation*}
   Q= y_1y_2x_0^2+y_0y_2x_1^2+y_0y_1x_2^2+F(y_0,y_1,y_2)x_3^2=\langle y_1y_2,y_0y_2,y_0y_1,F\rangle,
\end{equation*}
where $$F(y_0,y_1,y_2)=y_0^2+y_1^2+y_2^2-2(y_0y_1+y_0y_2+y_1y_2).$$
\end{example}

\begin{remark}\label{Remark:sing_of_HPT}
  Let the notation be as in Example \ref{example_HPT}. 

    The singularities of $Q$ have been carefully studied in \cite[\S 5.1]{HassetPirutkaTschinkel2018} and are given by the union of the six conics 
\begin{align*}
C_{y_1} &= \left\{\, y_1 = x_1 = x_3 = 0,\; y_2 x_0^{2} + y_0 x_2^{2} = 0 \,\right\}, \\
C_{y_0} &= \left\{\, y_0 = x_0 = x_3 = 0,\; y_2 x_1^{2} + y_1 x_2^{2} = 0 \,\right\}, \\
C_{y_2} &= \left\{\, y_2 = x_2 = x_3 = 0,\; y_0 x_1^{2} + y_1 x_0^{2} = 0 \,\right\}, \\
R_{y_1} &= \left\{\, y_0 - y_2 = y_1 = x_1 = 0,\; x_0^{2} + x_2^{2} - 4 x_3^{2} = 0 \,\right\}, \\
R_{y_0} &= \left\{\, y_1 - y_2 = y_0 = x_0 = 0,\; x_1^{2} + x_2^{2} - 4 x_3^{2} = 0 \,\right\}, \\
R_{y_2} &= \left\{\, y_0 - y_1 = y_2 = x_2 = 0,\; x_0^{2} + x_1^{2} - 4 x_3^{2} = 0 \,\right\}.
\end{align*}
In particular we have $$Q^{\mathrm{sing}}\subset \{y_0y_1y_2=0\}.$$  
\end{remark}

\begin{proposition}\label{proposition_rel_tor_ord_HPT}
    Let the notation be as in Example \ref{example_HPT}.
Let $$W=\{y_0y_1y_2\cdot x_0x_1x_2x_3=0\}.$$
Then $\Tor^\Lambda(Q,W)=2.$
\end{proposition}
\begin{proof}
    Consider the projection $f:Q\to \bb P^2_k$, $K=k(\frac{y_1}{y_0},\frac{y_2}{y_0})=k(\bb P^2_k)$ and the class
    $$\alpha=(\frac{y_1}{y_0},\frac{y_2}{y_0})\in Br(K)[2]=H^2_{}(K,\bb Z/2).$$
     By \cite[Prop. 10]{HassetPirutkaTschinkel2018}, $0\neq f^*\alpha\in H^2_{nr}(k(Q)/k,\bb Z/2)$ is unramified and nontrivial. Indeed, let   
     $$\psi:=\langle\langle \frac{y_1}{y_0},\frac{y_2}{y_0}\rangle\rangle=\langle 1,-\frac{y_1}{y_0},-\frac{y_2}{y_0},\frac{y_1}{y_0}\frac{y_2}{y_0}\rangle.$$
    Write $\beta:=f^*\alpha$.
    %, by \cite[Thm. 2.3]{SchreiederAMS}, $\beta=0$ iff $\frac{Q}{y_0}$ is a subform of $\psi$ which is the case iff (see \cite[Lem. 4.3]{SchreiederAMS}) $\frac{F}{y_0}$ is a square in $K$. 
    We want to apply Proposition \ref{unramifiedobstruction_relativedecomposition} using the class $\beta$ so let $\tau:Q'\to Q$ be an alteration of degree different from $2$. In order to conclude $\Tor(Q,W)>1$ we need to verify that for any subvarieties $E\subset Q'$ with $\tau(E)\subset W$ the restriction $(\tau^*\beta)|_E=0\in H^i(k(E)/k,\bb Z/2)$ vanishes. We distinguish the following two cases by considering the image $\tau (E)$ via the projection $f$ and treat them separately.
     \begin{caselist}
         \item Let $\overline{f\tau (E)}\neq\bb P^2_k$. \\
         Since $E$ does not dominate $\bb P^2_k$ and $Q'$ is smooth, the vanishing $(\tau^*\beta)|_E$ follows from \cite[Thm. 5.1 \text{ or }Thm. 9.2]{SchreiederAMS} (where the argument is carried out in much greater generality). 
         \item Let $\overline{f\tau (E)}=\bb P^2_k$. \\
    First note that the variety $Z=\tau(E)$ is not contained in $Q^{\mathrm{sing}}$ as it dominates $\bb P^2_k$ via $f$ and $Q^{\mathrm{sing}}$ does not dominate $\bb P^2_k$ via $f$ by Remark \ref{Remark:sing_of_HPT}. Hence we can consider the restriction $\_|_Z:{ H^2_{nr}(k(Q)/k,\bb Z/2)}\to { H^2_{nr}(k(Z)/k,\bb Z/2)}$ \cite[Thm 4.1.1]{CTsurvey95} (see also \cite[Sec. 2.5]{Merkurjev2008}, in the language of which the morphisms $Z\to Q$, $E\to Q'$ and $E\to Z$ satisfy condition $(\diamond)$) 
    and by \cite[Prop. 2.13]{Merkurjev2008} fits in the commutative diagram 
    \begin{center}
\begin{tikzcd}
	{ H^2_{nr}(k(Q')/k,\bb Z/2)} & { H^2_{nr}(k(E)/k,\bb Z/2)} \\
	{ H^2_{nr}(k(Q)/k,\bb Z/2)} & { H^2_{nr}(k(Z)/k,\bb Z/2).}
	\arrow["{\_|_E}", from=1-1, to=1-2]
	\arrow["{\tau^*}", from=2-1, to=1-1]
	\arrow["{\_|_Z}", from=2-1, to=2-2]
	\arrow["{{\tau|_E}^*}"', from=2-2, to=1-2]
\end{tikzcd}
    \end{center}
        Since $Z$ is irreducible, it must be contained in a horizontal irreducible component $D_i:=\{x_i=0\}\cap Q\subset W$. As $Z$ is not contained in the singular locus of $({D_i})^{\mathrm{sing}}=Q^{\mathrm{sing}}\cap D_i$ (or note that $D_i\to \Proj k[y_0,y_1,y_2]$ is generically smooth and therefore $Z\to D_i$ satisfies condition $(\diamond)$) we can consider the restriction $\_|_Z':H^2_{nr}(k(D_i)/k,\bb Z/2)\to H^2_{nr}(k(Z)/k,\bb Z/2)$ and by \cite[Cor. 12.4]{Ro96} the restriction $\_|_Z$ factors over it (through the restriction $\_|_{D_i}:H^2_{nr}(k(Q)/k,\bb Z/2)\to H^2_{nr}(k(D_i)/k,\bb Z/2)$, which exists as $D_i\to Q$ satisfies $(\diamond)$).  The divisors $D_i$ are conic bundles over $\Proj k[y_0,y_1,y_2]$ defined by $\{L_i=0\}$, for $L_i:=Q|_{x_i=0}\in k(\frac{y_1}{y_0},\frac{y_2}{y_0})[x_0,...,\hat x_i,...,x_3]$. By \cite[Thm. 2.3]{SchreiederAMS} to conclude the proof it is enough to show that $L_i$ is similar to a subform of the Pfister form $\langle\langle \frac{y_2}{y_0},\frac{y_1}{y_0}\rangle\rangle$. This is shown in Lemma \ref{Lemma_subform_of_Pfister}.
     \end{caselist}
\end{proof}
For the following observation see also the beginning of \cite[Sec. 6]{SchreiederAMS}. Denote $f:=\frac{F(y_0,y_1,y_2)}{y_0^2}$.
\begin{lemma}\label{Lemma_subform_of_Pfister}
    Over the field $K:=k(\frac{y_1}{y_0},\frac{y_2}{y_0})=k(\bb P^2)$ the forms  
    \begin{align*}
    L_0=\langle \frac{y_2}{y_0},\frac{y_1}{y_0},f\rangle,\,\quad L_1=\langle \frac{y_1}{y_0},\frac{y_1}{y_0}\frac{y_2}{y_0},f\rangle,\quad L_2=\langle \frac{y_2}{y_0},\frac{y_1}{y_0}\frac{y_2}{y_0},f\rangle,\quad  L_3=\langle \frac{y_1}{y_0},\frac{y_2}{y_0},\frac{y_1}{y_0}\frac{y_2}{y_0}\rangle
    \end{align*}
    are subforms of the Pfister form $P=\langle\langle \frac{y_2}{y_0},\frac{y_1}{y_0}\rangle\rangle$.
\end{lemma}
\begin{proof}
First note that since $-1$ is a square in $K$ ($k$ is algebraically closed), we may ignore signs, which makes the case of $L_3$ trivial. By symmetry it is now enough to show that $L_0$ and $L_1$ are subforms of $P$. Denote $u=\frac{y_2}{y_0},v=\frac{y_1}{y_0}$, and $S=\langle u,v\rangle$, $S'=\langle vu,v\rangle$. $S,S'$ are subforms of $P$ and
$$L_0=S\oplus \langle f\rangle \text{ and } L_1=S'\oplus \langle f\rangle.$$ 
%$L,L'$ are direct sum of $S$ and $\langle f\rangle$, respectively $S'$ and $\langle f\rangle$. 
It therefore suffices to show that $\langle f\rangle$ is a subform of  $\langle 1,\frac{y_1}{y_0}\frac{y_2}{y_0}\rangle$ resp. $\langle 1,\frac{y_2}{y_0}\rangle$. 
By the representation criterion \cite[I.2.3]{Lam2005quadraticForms} to show this it suffices to find $(a,b),(a',b')\in K^2-\{(0,0)\}$ such that 
    \begin{equation*}
        a^2+uvb^2=f\,\,\,\,\text{and}\,\,\,a'^2+ub'^2=f.
    \end{equation*}

%a \\
%b
%\end{pmatrix}$ is the quadratic form $\langle f\rangle$.} 
Write $i$ for a root of $-1$ and set $a=u+v-1,\,\,b=b'=2i,\,\,a'=u-v+1$. Then 
    \begin{equation*}
        a^2+uvb^2=(u+v-1)^2-4uv=f\,\,\,\,\text{and}\,\,\,\,a'^2+ub'^2=(u-v+1)^2-4u=f.
    \end{equation*}
\end{proof}

\begin{remark}\label{rem:HPT_unirational}
    It follows from Proposition \ref{proposition_rel_tor_ord_HPT} that $2\mid\Tor (Q)$. By an application of \cite[Lemma 14]{Schreieder2019quadricbundles}, which gives a unirational parametrisation of $Q$ of degree two, and \cite[Cor. 7.10]{Schreiedersurvey} we conclude that $2=\Tor (Q)$.
\end{remark}

\subsection{Very general bidegree hypersurfaces}
Let $k$ be a field and $\Lambda $ a ring. 
\begin{definition}\cite[Definition 2.11]{LangeZhang}\label{Def:very_general}
Let $\pi:\cal X\to B$ be a proper flat morphism of $k$-schemes with $B$ geometrically integral. A closed point $b\in B$ is very general with respect to $\pi$ if there exist field extensions $K/\kappa(b)$ and $L/k(B)$ and a field isomorphism $K\to L$ that induces an isomorphism of $\bb Z$-schemes
\begin{equation*}
    \cal X_b\times_{\kappa(b)}K\cong \cal X\times_{k(B)}L.
\end{equation*}
\end{definition}
We are interested in the case where $B$ is the parameter space of bidegree $(d,f)$ hypersurfaces in $\bb P^{n-r}_k\times \bb P^{r+1}_k$ with $n\geq r+2\geq 3$, hence
\begin{equation*}
    B=\bb P(H^0(\bb P^{n-r}_k,\cal O_{\bb P^{n-r}_k}(d))\otimes H^0(\bb P^{r+1}_k,\cal O_{\bb P^{r+1}_k}(f)))
\end{equation*}
and $\pi:\cal X\to B$ is the restriction of the projection $p_1:B\times \bb P^{n-r}_k\times \bb P^{r+1}_k\to B$ to the universal family of bidegree $(d,f)$ hypersurfaces. For a very general $b\in B$ we will call the fibre $\pi^{-1}(b)$ a very general bidegree hypersurface. Write $N:=\binom{n-r+d}{d}\binom{r+1+f}{f}-1$ for the dimension of $B$ and denote the coordinates on $B$ by $(t_0:...:t_N)$. 
The following gives concrete examples of very general hypersurfaces.
\begin{lemma}\label{Lemma:very_general_examples}
Let $k$ be an uncountable algebraically closed field with prime field $k_0$, then for every point $a=(a_0:...:a_N)\in B(k)$ such that 
\begin{equation}\label{very_transcendental}
P(a)\neq 0\,\,\text{for every homogeneous polynomial}\,\,P\in k_0[t_0,...,t_N]-\{0\}
\end{equation}
the associated hypersurface $\cal X_a\subset \bb P^{n-r}\times \bb P^{r+1}$ is very general. Moreover, for any points $a,b\in B(k)$ that satisfy the above condition (\ref{very_transcendental}) and any closed subscheme $W\subset \bb P^{n-r}\times \bb P^{r+1}$ that is defined over $k_0$, we have $$\Tor^\Lambda(\cal X_a, W\cap \cal X_a)=\Tor^\Lambda(\cal X_b, W\cap \cal X_b).$$ 
\end{lemma}
\begin{proof}
    For the first claim we follow the proof of \cite[Lemma 2.1]{Vial2013}. Denote the quotient field of $B$ by $k(B)=k(u_1,...,u_N)$ with $u_i=t_i/t_0$, then as $\overline{k(B)}$ and $k$ are both algebraically closed fields of the same transcendence degree, there exists a field isomorphism $\phi_a:\overline{k(B)}\cong k$ which sends the tuple $(u_1,...,u_N)$ to the tuple $(a_1/a_0,...,a_N/a_0)$ (here assumption $(\ref{very_transcendental})$ is used to ensure that $(a_1/a_0,...,a_N/a_0)$ is algebraically independent over $k_0$). It follows by construction that the base change $(\cal X_{\overline{k(B)}})_k$ along $\phi_a$ defines the same $k$-scheme as $\cal X _a$. Hence $\cal X _a$ is very general with respect to $\pi:\cal X\to B$ (with $K=k$ and $L=\overline{k(B)}$).  To prove the second claim note that the isomorphisms $\phi_a^{-1}$,$\phi_b^{-1}$ preserve the prime field $k_0$ hence 
    $$\Tor^\Lambda({\cal X_a}\times{\overline{k(B)}},( W\cap {\cal X_a})\times{\overline{k(B)}})=\Tor^\Lambda({\cal X_b}\times{\overline{k(B)}},(W\cap {\cal X_b})\times{\overline{k(B)}}).$$ 
   Indeed, there is an isomorphism (given by pushforward and pullback) \[\CH_0(\cal X_a\times\overline{k(B)}_{k(\cal X_a\times\overline{k(B)})})\cong \CH_0(\cal X_b\times\overline{k(B)}_{k(\cal X_b\times\overline{k(B)})})\] which preserves $W$ and the diagonal since they are defined over the prime field. 
    Finally, we conclude by Lemma \ref{Lemma_LS_3.6}(e).
\end{proof}
\begin{remark}
    Note that the locus of points in $B(k)$ which parametrize singular fibers is defined over the prime field $k_0\subset k$. To see this let $B',\cal X'$ be $k_0$-varieties and $f':\cal X'\to B'$ a $k_0$-morphism such that $B'_k=B$, $\cal X'_k=\cal X$ and $f'_k=f$.
The locus of points in $B$ which parametrize singular varieties is the schematic image under $f$ of the closed subscheme
\begin{equation*}
    S=\{(x,P)\,|\,P(x)=0\,\,\text{and}\,\,\mathrm(Jac \,P)_x\,\text{does not have full rank}\}
\end{equation*}
which is defined over $\bb Z$. As $S$ is clearly defined over $k_0$, we see by \cite[\href{https://stacks.math.columbia.edu/tag/0H4M}{Tag 081I}]{stacks-project} that $f(S)$ is defined over $k_0$. We conclude that any bidegree hypersurface which satisfies condition (\ref{very_transcendental}) is smooth.
\end{remark}

The key property of very general hypersurfaces is their flexibility in degeneration arguments.  
\begin{lemma}\label{Lemma:very_general_relative_tor}
   Let $k$ be an uncountable algebraically closed field with prime field $k_0$. Assume that there exists an integral bidegree $(d,f)$ hypersurface $X$ and a closed subscheme $W\subset \bb P^{n-r}_k\times \bb P^{r+1}_k$ that is defined over $k_0$, satisfies $\dim W\leq n-1$ and for which $\Tor^\Lambda(X, W\cap X)>1$.  Then for any $b\in B(k)$ that satisfies condition (\ref{very_transcendental}) we have $\Tor^\Lambda(\cal X_b,\cal X_b\cap W)>1$.
\end{lemma}
\begin{proof}
By Lemma \ref{Lemma:very_general_examples} it is enough to find one such point for which the claim holds, so let $b\in B(k)$ be a point that satisfies condition (\ref{very_transcendental}) and is different from the point $a\in B(k)$ associated to $X$.  Let $L\subset B$ be the line that passes through both $b$ and $a$. Let $R=\cal O_{L,a}$ be the local ring at $a$, then the special fiber of $\cal X_R$ is $\cal X_a$. Write $k(L)=k(u)$ and let $\{M_i(y,x)\}$ be the family of bidegree $(d,f)$ monomials, then we can identify the fiber 
\begin{equation*}
    \pi^{-1}(\eta_L)=\{\sum_{i=0}^{N}(a_i+ub_i)M_i(x,y)=0\}\subset \bb P^{n-r}_{k(u)}\times \bb P^{r+1}_{k(u)}
\end{equation*}
As $b_0\neq 0$ we can write $\pi^{-1}(\eta_L)=\{M_0+\sum_{i=1}^{N}v_iM_i(x,y)=0\}$ for $v_i=\frac{a_i+ub_i}{a_0+ub_0}\in k(u)$. Now write $\lambda_i:=\frac{b_i}{b_0}$, as both $\{\lambda_i\}$ and $\{v_i\}$ form algebraically independent families over $k_0$, it follows as in the proof of Lemma \ref{Lemma:very_general_examples} that $$\Tor^\Lambda(\cal X_b, W\cap \cal X_b)=\Tor^\Lambda(\cal X_{\overline{k(L)}}, W_{\overline{k(L)}}\cap \cal X_{\overline{k(L)}})$$
and we can conclude by \cite[Lem. 3.8]{LangeSchreieder2024} that $1<\Tor^\Lambda(\cal X_a, W\cap \cal X_a)\,|\,\Tor^\Lambda(\cal X_b, W\cap \cal X_b)$.
\end{proof}
Finally, we have the following more general statement. 
\begin{lemma}\label{Lemma:very_general_decomposition}
    Let $k$ be an uncountable algebraically closed field. Assume that there exists an integral bidegree $(d,f)$ hypersurface $X$ and a nonempty closed subscheme $W\subset \bb P^{n-r}_k\times \bb P^{r+1}_k$ for which $\Tor^\Lambda(X, W\cap X)>1$.  Then for any very general $b\in B(k)$ the $k$-variety $\cal X_b$ does not admit a decomposition of the diagonal.
\end{lemma}
\begin{proof}
    The claim follows by the same argument as at the end of the proof of \cite[Theorem 5.7]{LangeZhang}. 
Note that by assumption we have $\min\{n-r,r+1\}\geq 2$ and $d,f\geq 0$, hence $H^{0}(\cal X_b, \Omega^1)=0$. 
\end{proof}
\begin{remark}
   Let $k$ be algebraically closed field of characteristic zero. It follows by a result of Nicaise--Shinder \cite{NicaiseShinder2019} that if there exists a smooth $k$-variety $\cal X_a$ that is stably irrational, then a very general $k$-variety $\cal X_b$ is stably irrational. In this case one cannot drop the smoothness assumption for $\cal X_a$ as in the above Lemma. In fact a very general bidegree $(2,2)$ hypersurface in $\bb P^1\times \bb P^2$ is rational and degenerates to a cone over an elliptic curve, which is stably irrational but admits a decomposition of the diagonal \cite[Lem. 3.16]{LangeSkauli2023}. I.o.w., for a very general fiber we can find obstructions to the existence of the diagonal, but not to stable irrationality, in a singular fiber (not having the given property).
\end{remark}
\begin{remark}
    Statements analogous to Lemma \ref{Lemma:very_general_relative_tor} and Lemma \ref{Lemma:very_general_decomposition} hold for complete intersections of bidegree hypersurfaces as the parameter space of these is an open of product of projective spaces. 
\end{remark}
\section{Fourfolds of bidegree $(2,3)$ in $\bb P^1\times \bb P^4$}\label{Section_foufolds_bidegree}
It is shown in \cite[Theorem 6.3]{NicaiseOttem2022} that over a field $k$ of characteristic zero a very general hypersurface of bidegree $(2,3)$ in $\bb P_k^1\times \bb P_k^4$ is stably irrational. We upgrade this to obstruct the decomposition of the diagonal over arbitrary fields of characteristic different from two.
\begin{proposition}\label{prop:proposition_(2,3)_in_P^1timesP^4}
    Let $k$ be a field of characteristic different from two. A very general hypersurface of bidegree $(2,3)$ in $\bb P_k^1\times \bb P_k^4$ does not admit a decomposition of the diagonal, hence is not retract rational.
\end{proposition}
\begin{proof}
    We can assume that $k$ is algebraically closed and uncountable. Consider the integral bidegree $(2,3)$ hypersurface \begin{equation}
        X:=\{z_3^3 s_0^2
+ (z_4 z_0^2 + z_3 z_1^2 -2 z_3^2(z_3 +z_4))s_0s_1+(z_4 z_2^2+z_3(z_3-z_4)^2) s_1^2=0\}\end{equation}\[\subset \Proj{k[s_0,s_1]}\times \Proj {k[z_0,z_1,z_2,z_3,z_4]}.
     \]
    We see that the dense open subset $D_+(s_1)\times D_+(z_3)\cap X$ is isomorphic to the open $D_+(y_0)\times D_+(x_3)\cap Q$ for $Q$ as in Example \ref{example_HPT}. More precisely, under the ring isomorphism $\phi$ given by \[
    \frac{y_1}{y_0}\mapsto \frac{z_4}{z_3},\qquad
\frac{y_2}{y_0}\mapsto \frac{s_0}{s_1},\qquad
\frac{x_i}{x_3}\mapsto \frac{z_i}{z_3}\ \ (i=0,1,2)
    \]
    we have \[
    \frac{y_1}{y_0}\frac{y_2}{y_0}\left(\frac{x_0}{x_3}\right)^2
+\frac{y_2}{y_0}\left(\frac{x_1}{x_3}\right)^2
+\frac{y_1}{y_0}\left(\frac{x_2}{x_3}\right)^2
+F\!\left(1,\frac{y_1}{y_0},\frac{y_2}{y_0}\right)\mapsto \] \[\left(\frac{s_0}{s_1}\right)^2
+\left(
\frac{z_4}{z_3}\left(\frac{z_0}{z_3}\right)^2
+\left(\frac{z_1}{z_3}\right)^2
-2\left(1+\frac{z_4}{z_3}\right)
\right)\left(\frac{s_0}{s_1}\right)
+\frac{z_4}{z_3}\left(\frac{z_2}{z_3}\right)^2
+\left(1-\frac{z_4}{z_3}\right)^2.
    \]
    It follows from Proposition \ref{proposition_rel_tor_ord_HPT} that for $\Lambda:=\bb Z/2$ we have $\Tor^\Lambda(X,\{s_1z_3=0\})=2$. By setting $W=\{s_1z_3=0\}$ in Lemma \ref{Lemma:very_general_decomposition} we can conclude the proof.
\end{proof}
\section{Küchle varieties of type $b_4$}\label{Section_Küchle_varieties}
In this section, we sketch an application of the degeneration method of Lange--Schreieder and show that a very general Küchle variety of type $b_4$ does not admit a decomposition of the diagonal. Closely related arguments are carried out in full detail in Sections~\ref{Section:maintheorem} and~\ref{Sect:Raising_dimension_without_raising_degree}.\par
For the definition of a Küchle variety of type $b_4$ and its importance in the classification of Fano varieties we refer to \cite{Kuechle1995} and \cite{KuznetsovKuechle}. We will just use the fact that by a result of Kuznetsov \cite[Prop. 2.1]{KuznetsovKuechle} a smooth Küchle variety of type $b_4$ is isomorphic to a complete intersection $X$ of two hypersurfaces of bidegree $(1,1)$ and $(2,2)$ in $\bb P^3_{\bb C}\times \bb P^3_{\bb C}$. It follows by adjunction that the anticanonical sheaf of $X$ is $\omega_{X}^\vee\cong \cal O_X(1,1)$. Hence $X$ is Fano, and its torsion order is finite $\Tor^\bb Z(X)<\infty.$
Moe has shown that a very general intersection of such hypersurfaces is not stably rational \cite[Thm. 6.2.2]{Moethesis2020}.\par  Let $k$ be a field. Let $F$ define a very general bidegree $(2,2)$ hypersurface in $\Proj k[y_0,y_1,y_2]\times \Proj k[x_0,x_1,x_2]$. Let $R=k[t]_{(t)}$ and consider the degeneration 
    \begin{equation}\label{eq:Küchle_degeneration}
        \cal X:=\{F+y_0x_0^2a+y_1^2x_1b=0=ab+tx_0y_0\}\subset \Proj R[y_0,y_1,y_2,a]\times \Proj R[x_0,x_1,x_2,b].
    \end{equation}
     of a Küchle variety of type $b_4$ to the union of two $(2,2)$ hypersurfaces $Y_0$ and $Y_1$ in $\bb P^3_k\times \bb P^3_k$ that intersect in the $k$-variety $Z:=\{F=a=b=0\}$. The key ingredient of the degeneration method is the following theorem of Lange--Schreieder (building on work of Schreieder--Pavic \cite{PavicSchreieder2023}). 
 \begin{theorem}\label{Theorem_LS_4.3} \cite[Thm. 4.3]{LangeSchreieder2024}
    Let $R$ be a discrete valuation ring with fraction field $K$ and algebraically closed residue field $k$, and let $\Lambda$ be a ring of positive characteristic $c\in \bb Z_{\geq 1}$ such that the exponential characteristic of $k$ is invertible in $\Lambda$. Let $\cal X\to \Spec R$ be a proper flat separated $R$-scheme with geometrically integral generic fiber $X$ and special fiber $Y$. Now if $W_\cal X\subset \cal X$ is a closed subscheme such that
\begin{enumerate}
    \item $\cal X^\circ:= \cal X\setminus W_\cal X$ is a strictly semistable $R$-scheme, and
    \item $Y^\circ:=Y\setminus W_Y=\bigcup_{i\in I}Y_i^\circ$ has no triple intersections,
\end{enumerate}
then the group \begin{equation}\label{eq:LS_cocker}
    \mathrm{coker} \left(\Psi_{Y_L^\circ}^\Lambda:\bigoplus_{l\in I}\CH_1(Y_{l}^\circ\times_k L,\Lambda)\to \bigoplus_{\stackrel{i,j\in I}{i<j}}\CH_0(Y_{ij}^\circ\times_k L,\Lambda)\right) \end{equation}
    is $\Tor^\Lambda(\bar{X},W_{\bar{X}})$-torsion for every field extension $L/k$. 
    \end{theorem}
The equations defining (\ref{eq:Küchle_degeneration}) are chosen so that, after removing the coordinate hyperplanes, the groups of $1$-cycles on the components of the special fiber vanish as shown in Lemma~\ref{Lemma_simple_CH_1}. This makes the cokernel in \eqref{eq:LS_cocker} more tractable.
\begin{lemma}\label{Lemma_simple_CH_1}
Let $k$ be an algebraically closed field and
\[
X \subset \bb P^{n-r}_k \times \bb P^{r+1}_k
= \Proj k[y_0,\dots,y_{n-r}] \times \Proj k[x_0,\dots,x_{r+1}]
\]
a bidegree $(d,f)$ hypersurface given by the equation
\[
X=\Bigl\{\,Q(x_0,\dots,x_r,y_0,\dots,y_{n-r}) \;+\; P(x_0,\dots,x_r,y_0,\dots,y_{n-r})\,x_{r+1}=0\,\Bigr\},
\]
with $P\neq 0$. If $\{P=0\}$ contains a divisor of type $(1,0)$ and a divisor of type $(0,1)$ in
$\bb P^{n-r}_k \times \bb P^{r+1}_k$, then
\[
\CH_1\bigl(X\setminus\{P=0\}\bigr)\cong 0.
\]
\end{lemma}
\begin{proof}
The lemma follows from the fact that $X\setminus\{P=0\}$ is isomorphic to an open subscheme of
$\bb A^n_k$, and for such schemes the group $\CH_1$ vanishes. For the first fact consider the projection
\[
\pi \colon X \dashrightarrow \bb P^{n-r}_k \times \bb P^r_k,\qquad
(x_0:\cdots:x_{r+1},\, y_0:\cdots:y_{n-r}) \longmapsto (x_0:\cdots:x_r,\, y_0:\cdots:y_{n-r}).
\]
Restricting $\pi$ to
\[
\pi\big|_{X\setminus\{P=0\}} \colon X\setminus\{P=0\} \to
\bigl(\bb P^{n-r}_k \times \bb P^r_k\bigr)\setminus\{P=0\}
\]
gives an isomorphism with inverse
\[
\phi \colon \bigl(\bb P^{n-r}_k \times \bb P^r_k\bigr)\setminus\{P=0\} \to X\setminus\{P=0\},
\]
\[
(x_0:\cdots:x_r,\, y_0:\cdots:y_{n-r}) \longmapsto
\Bigl(x_0:\cdots:x_r:\,-
\frac{Q(x_0,\dots,x_r,y_0,\dots,y_{n-r})}{P(x_0,\dots,x_r,y_0,\dots,y_{n-r})}
:\, y_0:\cdots:y_{n-r}\Bigr).
\]
\end{proof}

\begin{theorem}
   Let $k$ be an algebraically closed field of characteristic zero. Then a very general complete intersection of hypersurfaces of bidegree $(1,1)$ and $(2,2)$ in $\bb P^3_k\times \bb P^3_k$ does not admit a decomposition of the diagonal. In particular, a very general Küchle variety of type $b_4$ is not retract rational.
\end{theorem}
\begin{proof}
    Let $C$ be a very general complete intersection of hypersurfaces of bidegree $(1,1)$ and $(2,2)$ in $\bb P^3_k\times \bb P^3_k$ and $\Lambda=\bb Z/\Tor^\bb Z(C)$.  Note that the geometric generic fiber of the family (\ref{eq:Küchle_degeneration}), denoted by $\bar{X}:=\cal X_{\overline{k(t)}}$, is integral. 
    Let
    \begin{equation*}
        W_{\cal X}:=\{x_0x_1y_0y_1=0\}\subset \cal X.
    \end{equation*}
    Since $\cal X^{\mathrm{sing}}\subset W_{\cal X}$ and $Y_0\setminus W_{\cal X},Y_1-W_{\cal X}$ and $Z$ are smooth, we get that $\cal X^\circ$ is strictly semistable. Therefore the assumptions of \ref{Theorem_LS_4.3} are satisfied and for $L=k(Z)$ the cokernel of the map 
    $$\Psi_{Y_L^\circ}^\Lambda:\CH_1(Y_0^\circ\times_k L,\Lambda)\oplus \CH_1(Y_1^\circ\times_k L,\Lambda)\to \CH_0(Z^\circ\times_k L,\Lambda)$$
    is $\Tor^\Lambda(X^{\circ},\emptyset)$-torsion. By Lemma \ref{Lemma_simple_CH_1} we have that $\mathrm{coker}(\Psi_{Y_L^\circ}^\Lambda)=\CH_0(Z^\circ\times_k L,\Lambda)$ and it now follows from Proposition \ref{Proposition:Tor_simple_subschemes} that $\Tor^\Lambda (Z,W_{\cal X}\cap Z)>1$. The assertion now follows from the fact that $\Tor^\Lambda (Z,W_{\cal X}\cap Z)|\Tor^\Lambda(\bar{X},W_{\bar{X}})|\Tor^\bb Z(\bar{X},W_{\bar{X}})|\Tor^\bb Z(\bar{X})$, where the last divisibility uses that $\bar{X}$ is Fano.
\end{proof}

%%%%%%%%%%%%%%%%%%%

\section{Raising dimension and degree}\label{Section:maintheorem}

In order to prove our main theorems we need a general degeneration statement that can be applied iteratively starting from a given explicit example and which raises both the degree and the dimension. 
%This will be developed next.
%A general degeneration statement for bidegree hypersurfaces.

\theoremraisedegdim*

We begin by proving the special case of just raising the degree, i.e. $n'=n$.
\begin{proposition}\label{proposition_raise_degree}
   Theorem \ref{theorem_raise_deg_dim} holds for $n=n'$.
\end{proposition}
\begin{proof}
Let $k$ be an uncountable and algebraically closed field of characteristic different from two and let $R=k[t]_{(t)}$.  We consider the degeneration 
 $$\cal X=\{tH-G\cdot F\}\subset \Proj R[y_0,...,y_{n-r}]\times \Proj R[x_0,...,x_{r+1}],$$
 where $H$ is a very general bidegree $(d',f')$ hypersurface, $G$ is a very general bidegree $(d,f)$ hypersurface and $F$ is a bidegree $(d'-d,f'-f)$ hypersurface which is not contained in $W_{\cal X}=\{y_0x_0y_1x_1=0\}$ 
 Let $X$ be the generic fiber and $Y$ the special fiber of $\cal X$. 
If we choose $F=y_0^{d'-d}x_0^{f'-f}$, $G$ to satisfy condition (\ref{very_transcendental}) and $H$ to satisfy condition (\ref{very_transcendental}) over the extension $k_0(G)$ (given by the coefficients of $G$), which we can as the field $k$ is uncountable hence has uncountable transcendence degree over its prime field $k_0$. Then for $t$ transcendental over $k$ the generic fiber $tH-Gy_0^{d'-d}x_0^{f'-f}$ satisfies condition (\ref{very_transcendental}) over $k_0$, hence is smooth and integral. Then up to base change (see for example \cite[Proof of Lem. 3.8]{LangeSchreieder2024}) we may assume that $\Tor^\Lambda(X,W)=\Tor^\Lambda(\bar X,W_{\bar X})$. 
Let $Y^\circ:=Y\setminus W_Y$ and $X^\circ:=X\setminus W_{X}$.
 Let $A=\roi_{\cal X,\eta}$, where $\eta$ is the generic point of $Y^\circ$. Then the specialisation map
$$sp_{Y^\circ}:\CH_0(X^\circ_{K(X)})\to \CH_0(Y^\circ_{k(Y^\circ)})$$
 maps the diagonal point $\delta_X$ to the diagonal point of $Y^\circ$.
%Indeed, the map factorises through $\CH_0(Y^\circ_{K(Y)})$, the second map being induced by the pullback along (open) immersions of open subsets of $Y_i$. 
This implies the claim that $\Tor^\Lambda(\{G=0\},\{y_0x_0y_1x_1\})=\Tor^\Lambda(Y,W)|\Tor^\Lambda(X,W)$.
\end{proof}
\begin{remark}
    Proposition \ref{proposition_raise_degree} is a relative version of \cite[Lem. 2.4]{Totaro2016} based on \cite[Lem. 3.8]{LangeSchreieder2024}.
\end{remark}
We will prove the remaining, and more difficult, cases, raising the dimension, of Theorem \ref{theorem_raise_deg_dim} using the following ``double cone" construction. 
\begin{definition}\label{Def:double_cone_for_deg_dim}
 Let $k$ be an algebraically closed field of characteristic different from two and let $R=k[t]_{(t)}$. 
Consider the bidegree $(2,0)$ polynomial
\begin{equation*}
    \{F:=ab+ty_0^2=0\}\subset \Proj R[y_0,...,y_{n-r},a,b]\times \Proj R[x_0,...,x_{r+1}]
\end{equation*}
and for $d'= d+1\geq 3$ and $f\geq 2$ denote the bidegree $(d,f)$ polynomial 
\begin{equation*}
    \{H:=G+y_0^{d-1}x_0^{f}a+y_1^{d-1}x_1^{f}b=0\} \subset \Proj k[y_0,...,y_{n-r},a,b]\times \Proj k[x_0,...,x_{r+1}].
\end{equation*}
for $G$ a bidegree $(d,f)$ hypersurface in $\Proj k[y_0,...,y_{n-r}]\times \Proj k[x_0,...,x_{r+1}]$. Let
\begin{equation*}
    \cal X:=\{H=F=0\}\subset \Proj R[y_0,...,y_{n-r},a,b]\times \Proj R[x_0,...,x_{r+1}]
\end{equation*}
be the degeneration given by the complete intersection of the bidegree hypersurfaces defined by the vanishing of $H$ and $F$ in $\bb P_R^{n-r+2}\times \bb P_R^{r+1}$.
Let $K=\mathrm{Quot}(R)$. We denote the generic fiber by $X=\cal X\times K$. 
The special fiber $Y=\cal X\times _R k$ has the following two components:
    \begin{align*}
        Y_0:=\{G+y_0^{d-1}x_0^{f}a=0\}\subset \{b=0\}=\bb P_k^{n-r+1}\times \bb P^{r+1}_k,\\
        Y_1:=\{G+y_1^{d-1}x_1^{f}b=0\}\subset \{a=0\}=\bb P_k^{n-r+1}\times \bb P^{r+1}_k.
    \end{align*}
    The intersection $Z:=Y_0\cap Y_1$ is the bidegree $(d,f)$ hypersurface
\begin{equation*}
    Z:=\{G=0\}\subset \{a=b=0\}= \bb P^{n-r}_k\times \bb P^{r+1}_k.
\end{equation*}
\end{definition}

\begin{lemma}\label{lemma_generic_fiber_integral_birational}
Assume that $G$ satisfies the condition (\ref{very_transcendental}) of Lemma \ref{Lemma:very_general_examples}, then the geometric generic fiber $\bar X:= X_{\bar K}$, given by 
   \begin{equation*}
      \{H=F=0\}\subset \Proj \bar{K}[y_0,...,y_{n-r},a,b]\times \Proj \bar{K}[x_0,...,x_{r+1}],
   \end{equation*}
   is integral and birational to the bidegree $(d',f)$ hypersurface
   \begin{equation*}
      X':= \{Ga+y_0^{d-1}x_0^{f}a^2-ty_1^{d-1}x_1^{f}y_0^2=0\}\subset \Proj \bar{K}[y_0,...,y_{n-r},a]\times \Proj \bar{K}[x_0,...,x_{r+1}].
      \end{equation*}
\end{lemma}
Note that a general complete intersection is integral. However, $F$ is not general and we therefore need the following more complicated arguments in order to show integrality.
\begin{proof}
On the open set $\{a\neq 0\}\subset\bar{X}$ we have $ b=\frac{-{ty_0^2}}{a} $ which we substitute in $H$ and then homogenise the equation. This shows that $\bar{X}$ is birational to $X'$. We show that $X'$ is integral. 
Work on the principal open $U:=D_+(y_0)\times D_+(x_0)$,
and dehomogenize by $y_0=1=x_0$. Write
\[
y_i:=\frac{y_i}{y_0}\ (i\ge 1),\qquad x_j:=\frac{x_j}{x_0}\ (j\ge 1),\qquad a:=\frac{a}{y_0}.
\]
Let $G_U$ be the dehomogenization of $G$ on $U$ and set $A:=\bar K[y_1,\dots,y_{n-r},x_1,\dots,x_{r+1}]$ and $L:=\Frac(A)$.
Then on $U$ the equation of $X'$ becomes the monic quadratic $a^2+G_Ua-t\,y_1^{\,d-1}x_1^{\,f}\in A[a]$,
with discriminant
\[
\Delta:=G_U^2+4t\,y_1^{\,d-1}x_1^{\,f}\in A.
\]
If $\Delta$ were a square in $L$, then clearing denominators in the UFD $A$ shows $\Delta$ is a square in $A$. Consider the $\bar K$-algebra map $\varphi:A\to \bar K[x_1]$ given by
\[
y_1\mapsto 1,\quad y_i\mapsto 0\ (i\ge 2),\quad x_j\mapsto 0\ (j\neq 1).
\]
Since $G$ satisfies condition (\ref{very_transcendental}), we see that the coefficient of the monomials of $y_0^d x_1^f$ and $y_0^dx_1^{f-1}x_0$ in $G$ are nonzero. Assume that $\varphi(\Delta)=p(x_1)^2$ in $\bar K[x_1]$. Then $(\phi(G_U)-p(x_1))({\phi(G_U)+p(x_1)})=4t\,x_1^{\,f}$ implies that the only possible prime divisor of $p(x_1)+\phi(G_U)$ as well as of $p(x_1)-\phi(G_U)$ is $x_1$ and that the multiplicities of these add up to $f$. The equality \[
\phi(G_U)=\frac{\phi(G_U)-p(x_1)}{2}+\frac{\phi(G_U)+p(x_1)}{2}
\]
then gives a contradiction to $f\geq 2$ (otherwise $X$ would be rational). Indeed, the contradiction follows by comparing the coefficients of $x_1^f$ and $x_1^{f-1}$ in the two polynomials on both sides of the equation. 
$\phi(G_U)$ has nonzero coefficients for $x_1^f$ and $x_1^{f-1}$, but $\frac{\phi(G_U)-p(x_1)}{2}+\frac{\phi(G_U)+p(x_1)}{2}$ cannot since it is of the form $c_1x_1^{f-a}+c_2x_1^a$ with $c_1,c_2\in \bar K^\times$.

So $\Delta\notin L^{\times 2}$ and the quadratic is irreducible over $L$, hence $X'\cap U$ is integral. To see that $X'$ is integral we note that it is irreducible because every one of its components must meet $U$ and reduced because it is affine locally $S_1$ (see \cite[Lemma 10.104.2.]{stacks-project} and \cite[Lemma 28.12.3.]{stacks-project}) and $R_0$ because it is irreducible and contains a dense open integral subscheme.
Finally we can conclude that $\bar X$ is integral using the same arguments as for $X'$ as they are birational to each other.

\end{proof}
\begin{definition}\label{Def:closed_for_family_deg_dim}
    Let 
    \begin{equation*}
        W_\cal X:=\{y_0x_0y_1x_1=0\}\cap \cal X\subset \cal X.
    \end{equation*}
      Denote by $W_Y$ the special fiber and by $W_X$ (resp. $W_{\bar{X}}$) the generic (resp. geometric generic) fiber of $W_\cal X$. The complements are denoted as follows: \begin{equation*}
        Y^\circ:=Y\setminus W_Y,\,\,Y_i^\circ:=Y_i\setminus W_Y,\,\,Z^\circ:=Z\setminus W_Y,\,\,X^\circ:=X_K\setminus W_{X_K}.
    \end{equation*}
\end{definition}

\begin{lemma}\label{lemma_chow_generalcase}
    $\CH_1(Y_i^\circ\times_k L,\Lambda)=0$ for $i\in \{0,1\}$ and every field extension $L/k$.
\end{lemma}
\begin{proof}
    The claim follows from Lemma \ref{Lemma_simple_CH_1} as $\{x_0=0\}\cup \{y_0=0\}\subset W_X$.
\end{proof}
\begin{lemma}\label{Lemma_semistability_general}
    $\cal X ^\circ=\cal X-W_{\cal X}$ is strictly semistable.
\end{lemma}
\begin{proof}
Recall that $\cal X\subset \bb P_R^{n-r+2}\times \bb P_R^{r+1}$ is the complete intersection
\[
\cal X=\{F=0\}\cap\{H=0\},
\qquad
F=ab+ty_0^2,
\qquad
H=G+y_0^{d-1}x_0^f a+y_{1}^{d-1}x_1^f b,
\]
and that $W_{\cal X}=\{y_0x_0y_1x_1=0\}$. On $\cal X^\circ$ we have $y_0x_0y_1x_1\neq 0$, hence
$\cal X^\circ\subset D_+(y_0)\cap D_+(x_0)$. On the affine chart
\[
U:=\cal X\cap D_+(y_0)\cap D_+(x_0)
\]
we dehomogenize by $y_0=1=x_0$. Writing (by abuse of notation) the remaining affine coordinates as
\[
y_i=\frac{y_i}{y_0}\ (i\ge 1),\qquad x_j=\frac{x_j}{x_0}\ (j\ge 1),
\]
the equations of $\cal X\cap U$ become
\[
f:=ab+t=0,\qquad h:=G_U+a+y_1^{d-1}x_1^f\,b=0,
\]
where $G_U$ denotes the dehomogenization of $G$ on $U$.

We first check that $\cal X^\circ$ is regular. The Jacobian matrix of $(h,f)$ has the $2\times 2$ minor in the
columns $(a,t)$ equal to
\[
\det\begin{pmatrix}
\frac{\partial h}{\partial a} & \frac{\partial h}{\partial t}\\[2pt]
\frac{\partial f}{\partial a} & \frac{\partial f}{\partial t}
\end{pmatrix}
=
\det\begin{pmatrix}
1 & 0\\
b & 1
\end{pmatrix}
=1,
\]
so $\Jac(h,f)$ has rank $2$ everywhere on $\cal X\cap U$, hence $\cal X^\circ$ is regular. The special fiber $Y=\cal X\times_R k$ is given by $\{t=0\}$, and on $\cal X^\circ$ the relation $f=ab+t$ yields
\[
Y^\circ=\{t=0\}=\{ab=0\}=\{a=0\}\cup\{b=0\},
\]
so $Y^\circ=Y_0^\circ\cup Y_1^\circ$ with $Y_0^\circ=\{b=0\}$ and $Y_1^\circ=\{a=0\}$, both Cartier divisors on the
regular scheme $\cal X^\circ$. We check that $Y_0^\circ$ and $Y_1^\circ$ are smooth. On $Y_0^\circ$ we have $b=0$ and the remaining defining
equation is $h|_{b=0}=G_U+a=0$, with
\[
\frac{\partial}{\partial a}(h|_{b=0})=1\neq 0
\]
on $\cal X^\circ$, hence $Y_0^\circ$ is smooth. Similarly, on $Y_1^\circ$ we have $a=0$ and the remaining defining
equation is $h|_{a=0}=G_U+y_1^{d-1}x_1^f b=0$, with
\[
\frac{\partial}{\partial b}(h|_{a=0})=y_1^{d-1}x_1^f\neq 0
\]
on $\cal X^\circ$ since $y_1x_1\neq 0$ there. Therefore $Y^\circ$ is a strict normal crossings divisor on $\cal X^\circ$.
We can conclude that the morphism $\cal X^\circ\to \Spec(R)$ is strictly semistable \cite[Definition 1.1]{Hartl2001}.
\end{proof}

\begin{proof}[Proof of Theorem \ref{theorem_raise_deg_dim}]
By \cite[Lem. 2.13.(ii)]{LangeZhang} we can assume that $k$ is algebraically closed and uncountable. By Lemma \ref{Lemma:very_general_relative_tor} there exists a very general bidegree $(d,f)$ hypersurface $$\{G=0\}\subset \bb P^{n-r}_k\times \bb P^{r+1}_k=\Proj k[y_0,...,y_{n-r}]\times \Proj k[x_0,...,x_{r+1}]$$
such that $\Tor^{\Lambda}(\{G=0\},\{y_{0}x_{0}y_{1}x_{1}=0\})\neq 1$. We can now apply the construction of Definition \ref{Def:double_cone_for_deg_dim} and we obtain a family $\cal X$ over $R$ such that the intersection of the two components of the special fiber is $Z=\{G=0\}$. 
Set $W_\cal X$ as in Definition \ref{Def:closed_for_family_deg_dim}. Since $\cal X^\circ$ is strictly semistable by Lemma \ref{Lemma_semistability_general} and the generic fiber $X$ is geometrically integral by Lemma \ref{lemma_generic_fiber_integral_birational}, the assumptions of Theorem \ref{Theorem_LS_4.3} are satisfied and for $L=k(Z^\circ)$, loc. cit. tells us that the cokernel of the obstruction map 
$$\Psi_{Y_L^\circ}^\Lambda:\CH_1(Y_0^\circ\times_k L,\Lambda)\oplus \CH_1(Y_1^\circ\times_k L,\Lambda)\to \CH_0(Z^\circ\times_k L,\Lambda)$$
    is $\Tor^\Lambda(\bar X^{},W_{\bar{X}})$-torsion. 
    By Lemma \ref{lemma_chow_generalcase} we have that $\mathrm{coker}(\Psi_{Y_L^\circ}^\Lambda)=\CH_0(Z^\circ\times_k L,\Lambda)$, hence 
    $$ 1<\mathrm{ord} (\delta_{Z^\circ})\mid  \Tor^\Lambda(\bar X^{},W_{\bar{X}}).$$ Via the birational isomorphism of Lemma \ref{lemma_generic_fiber_integral_birational}, this implies that also $$1<\Tor^\Lambda( \bar X,\{y_0x_0y_1x_1=0\})=\Tor^\Lambda( X',\{y_0x_0y_1x_1a=0\}).$$ For this note that on the open $\{y_0\neq 0\}$ in the generic fiber we always have $a\neq 0$. As $X'$ is integral, the proof concludes with an application of Lemma \ref{Lemma:very_general_decomposition}.
\end{proof}

\section{Raising dimension without raising degree:Sixfolds of bidegree $(3,2)$ in $\bb P^4\times \bb P^3$}\label{Sect:Raising_dimension_without_raising_degree}

%Throughout this section let $k$ be an algebraically closed field of characteristic different from two.
In this section we prove the following theorem:
%\ThreeTwoTheorem*

\ThreeTwoTheorem*

The strategy of the proof of Theorem \ref{theorem_3_2_in_4_3} is as follows:
firstly, as always
\begin{enumerate}
    \item to specialise (in the more lax sense which includes lower dimensional strata) a $(3,2)$-$(2,0)$-complete intersection in $\bb P^5_R\times \bb P^3_R$ to a bidegree $(3,2)$ hypersurface (also called reference variety in the following) in $\bb P^3_k\times \bb P^3_k$ which does not admit a decomposition of the diagonal. Secondly,
\item find a reference bidegree $(3,2)$ hypersurface in $\bb P^3_k\times \bb P^3_k$ such that it is possible to increase the dimension of the bidegree hypersurface without increasing its degree by applying a ``substitution and coordinate change" to the generic fiber resulting from Step (i). This does not quite work if we specialise to $X'$ of the following lemma. Instead, we specialise to a variety which is birational to $X'$ and for which the substitution and coordinate change works (see Lemma \ref{lemma_substitute_coordinate_change}).
\end{enumerate}

From here on, let $k$ be an algebraically closed field of characteristic different from two.
\begin{lemma}\label{lemma_3_2_in_3_3}
    Let $Q$ be as in Example \ref{example_HPT}. Then the $k$-variety
     \begin{equation*}
      X':= \{Qa+y_1x_0^2a^2-ty_1x_2^2y_0^2=0\}\subset \Proj k[y_0,y_1,y_2,a]\times \Proj k[x_0,x_1,x_2,x_3]
      \end{equation*}
    does not admit a decomposition of the diagonal with respect to the closed subscheme 
    \begin{equation*}
        W_{X'}:=\{y_0y_1y_2x_0x_1x_2x_3a=0\}\subset X'.
    \end{equation*}
  %  In particular, a very general hypersurface of bidegree $(3,2)$ in $\bb P^3_k\times \bb P^3_k$ does not admit a decomposition of the diagonal.
\end{lemma}
\begin{proof}
    This follows from the arguments of Section \ref{Section:maintheorem} with $d=f=2$, $d'=3$, $G=Q$ and $Q$ as in  
    Example \ref{example_HPT} and $W=\{y_0y_1y_2x_0x_1x_2x_3\}$ (note that the proof of the last section goes through verbatim with this $W$ if we assume $n\geq r+1\geq 2$). The only difference is that $Q$ does not satisfy condition (\ref{very_transcendental}) of Lemma \ref{Lemma:very_general_examples} and therefore Lemma \ref{lemma_generic_fiber_integral_birational} cannot be applied. We prove the statement of Lemma \ref{lemma_generic_fiber_integral_birational} in this case. Work on the principal open
\[
U:=D_+(y_1)\times D_+(x_0)\subset \bb P^3_{\bar K}\times \bb P^3_{\bar K},
\]
and dehomogenize by $y_1=1=x_0$. Set $A:=\bar K[y_0,y_2,x_1,x_2,x_3]$, $ L:=\Frac(A)$, and let $Q_U:=Q(y_0,1,y_2;1,x_1,x_2,x_3)\in A$. Then on $U$ the equation of
$X'$ becomes the monic quadratic
\[
f(a):=a^2+Q_Ua-tx_2^2y_0^2\in A[a].
\]
Its discriminant is
\[
\Delta:=Q_U^2+4t\,x_2^2y_0^2\in A.
\]
To prove that $f$ is irreducible in $A$, it suffices to show $\Delta\notin L^{\times 2}$. Now apply the $\bar K$-algebra map
\[
\varphi:A\to \bar K[x_2],\qquad
y_0\mapsto 1,\ \ y_2\mapsto 0,\ \ x_1\mapsto 0,\ \ x_3\mapsto 0.
\]
For $Q$ as in Example \ref{example_HPT} one has $\varphi(Q_U)=x_2^2$, hence $\varphi(\Delta)=x_2^4+4t\,x_2^2=x_2^2(x_2^2+4t)$. Let $\beta\in\bar K$ satisfy $\beta^2+4t=0$; since $\operatorname{char}(\bar K)\neq 2$, the factor
$x_2^2+4t$ has a simple zero at $x_2=\beta$, so the valuation $v_\beta$ on
$\bar K(x_2)$ gives $v_\beta(\varphi(\Delta))=1$, which is odd. Thus
$\varphi(\Delta)$ is not a square in $\bar K(x_2)$.
\end{proof}

%We start with the construction of the latter bidegree $(3,2)$ hypersurface in $\bb P^3\times \bb P^3$.
The birational isomorphism in Step (ii) is given by the following lemma.
It is inspired by \cite[p. 2]{HassetPirutkaTschinkel2019} where  the argument is called ``multiply through and absorb the squares". We note that this argument is also used in \cite[Example 4.6]{LangeSchreieder2024}.
    \begin{lemma}\label{lemma_absorb_the_squares}
        The bidegree $(3,2)$ hypersurface of $\Proj k[y_0,y_1,y_2,a]\times \Proj k[x_0,x_1,x_2,x_3]$ defined by 
   \begin{equation*}
      X':= \{F':=Qa+y_1x_0^2a^2-ty_1x_2^2y_0^2=0\}
      \end{equation*}
      is birational to the $(3,2)$ hypersurface of $\Proj k[y_0,y_1,y_2,a]\times \Proj k[x_0,x_1,X_2,x_3]$ defined by 
       \begin{equation*}
      V:= \{G:=y_1^2(y_2+a)x_0^2+y_0y_2y_1x_1^2+ay_0(a-ty_0)X_2^2+y_1F(y_0,y_1,y_2)x_3^2=0\}
      \end{equation*}
      where $F$ is as in Example \ref{example_HPT}.
    \end{lemma}
\begin{proof}
    Consider the dehomogenisation of the equations defining $X'$ and $V$ with respect to $a$ and $x_3$ (i.e. the local equations in the open $D_+(a)\times D_+(x_3)$ of $\bb P^3_k\times \bb P^3_k$); we note that this is where we use the extra coordinate $a$ which we do not have for $(2,2)$-hypersurfaces in $\bb P^2_k\times \bb P^3_k$): setting $a=1=x_3$ in $F'$, we get
\begin{align*}
f:=F'(y_1,y_2,y_3,1,x_0,x_1,x_2,1)&=y_1y_2x_0^2+y_0y_2x_1^2+y_0y_1x_2^2+ y_1x_0^2-ty_1x_2^2y_0^2+F(y_0,y_1,y_2)\\
&=y_1(y_2+1)x_0^2+(y_0y_2)x_1^2+y_0y_1(1-ty_0)x_2^2+F(y_0,y_1,y_2).
\end{align*}
Setting $a=1=x_3$ in $G$, we get
$$g:=G(y_1,y_2,y_3,1,x_0,x_1,X_2,1)=y_1^2(y_2+1)x_0^2+y_0y_2y_1x_1^2+y_0(1-ty_0)X_2^2+y_1F(y_0,y_1,y_2).$$
  %  \begin{align*}
   %     f&=y_1(y_2+1)x_0^2+(y_0y_2)x_1^2+y_0y_1(1-ty_0)x_2^2+F(y_0,y_1,y_2),\\
        %g&=y_1^2(y_2+1)x_0^2+y_0y_2y_1x_1^2+y_0(1-ty_0)X_2^2+y_1F(y_0,y_1,y_2).
    %\end{align*}
    We now multiply $f$ through by $y_1$ and absorb the squares $y_1^2$ into $x_2^2$, i.o.w. we consider the isomorphism
    \begin{align*}
      \phi:\bar{K}[y_0,y_1,y_2,x_0,x_1,X_2,x_3,\frac{1}{y_1}] & \overset{\cong }{\to} \bar{K}[y_0,y_1,y_2,x_0,x_1,x_2,x_3,\frac{1}{y_1}] \\
      X_2&\to y_1 x_2
    \end{align*}
    which sends $\phi(g)=y_1f$. This implies that $X'$ and $V$ are isomorphic on the open subsets cut out by $a,x_3,y_1\neq 0$ and therefore birational. Finally, note that since $X'$ is integral by Lemma \ref{lemma_3_2_in_3_3}, $V$ is also integral since $\{ax_3y_1=0\}$ does not contain a component of $V$ and $V$ is Cohen-Macaulay and therefore equidimensional.

\end{proof}

\begin{corollary}
    $V$ does not admit a decomposition of the diagonal relative to the closed subset
    $$W=\{y_0y_1y_2x_0x_1X_2x_3a=0\}.$$
\end{corollary}
\begin{proof}
The claim follows by Lemma $\ref{lemma_absorb_the_squares}$ and Lemma \ref{lemma_3_2_in_3_3}.
\end{proof}

\begin{definition}
Let $s\neq t$ be a transcendental element over $k$, denote $S=k[s]_{(s)}$. We define the bidegree $(2,0)$ polynomial
\begin{equation*}
    F:=cd+sy_1^2=0\in  S[y_0,y_1,y_2,a,x_0,x_1,X_2,x_3,c,d]
\end{equation*}
and the bidegree $(3,2)$ polynomial
\begin{equation*}
    H:=G+y_1^2x_1^2c+y_0^2x_0^2d \in  S[y_0,y_1,y_2,a,c,d,x_0,x_1,X_2,x_3]
\end{equation*}
where $G$ is the bidegree $(3,2)$ polynomial defined in Lemma \ref{lemma_absorb_the_squares}.
Let
\begin{equation*}
    \cal X:=\{H=F=0\}\subset  \Proj S[y_0,y_1,y_2,a,c,d]\times \Proj S[x_0,x_1,X_2,x_3]
\end{equation*}
be the degeneration given by the complete intersection of the bidegree hypersurfaces defined by the vanishing of $H$ and $F$ in $\bb P_S^5\times \bb P_S^3$.
Let $K=\mathrm{Quot}(S)$. We denote the generic fiber by $X=\cal X\times K$ and the special fiber by $Y=\cal X\times _S k$. 
\end{definition}

The special fiber $Y$ has the following two components:
    \begin{align*}
        Y_0:=\{G+y_1^2x_1^2c=0\}\subset \{d=0\}=\bb P_k^{4}\times \bb P^3_k,\\
        Y_1:=\{G+y_0^2x_0^2d=0\}\subset \{c=0\}=\bb P_k^{4}\times \bb P^3_k.
    \end{align*}
    The intersection $V=Y_0\cap Y_1$ is the $(3,2)$ hypersurface
\begin{equation*}
    V=\{G=0\}\subset \{c=d=0\}= \bb P^{3}_k\times \bb P^3_k.
\end{equation*}
\begin{lemma}\label{lemma_substitute_coordinate_change}
The geometric generic fiber $\bar X:= X_{\bar K}$, given by 
   \begin{equation*}
      \{H=F=0\}\subset \Proj \bar{K}[y_0,y_1,y_2,a,c,d]\times \Proj \bar{K}[x_0,x_1,X_2,x_3],
   \end{equation*}
   is integral and birational to the bidegree $(3,2)$ hypersurface of $\Proj \bar{K}[y_0,y_1,y_2,a,c]\times \Proj \bar{K}[x_0,x_1,X_2,x_3]$
   \begin{equation*}
      T:= \{c^2(y_2+a)x_0^2+y_0y_2y_1x_1^2+ay_0(a-ty_0)X_2^2+y_1F(y_0,y_1,y_2)x_3^2+y_1^2x_1^2c-sy_0^2x_0^2c=0\}
      \end{equation*}
\end{lemma}
\begin{proof}
We can now apply the argument of \cite[Example 4.6]{LangeSchreieder2024} to increase the dimension without increasing the degree. 
We work on the open subset $\{cy_1\neq 0\}$. Substituting $$d=\frac{-sy_1^2}{c}$$
in $H$, we get $H=G+y_1^2x_1^2c+y_0^2x_0^2\cdot \frac{-sy_1^2}{c}$. We now make a coordinate change $x_0\mapsto \frac{cx_0}{y_1}$ and get 
$$H=c^2(y_2+a)x_0^2+y_0y_2y_1x_1^2+ay_0(a-ty_0)X_2^2+y_1F(y_0,y_1,y_2)x_3^2+y_1^2x_1^2c-sy_0^2x_0^2c.$$
As in the proof of Lemma \ref{lemma_generic_fiber_integral_birational} it is enough to prove that $T$ is integral on the open chart $U=D_+(y_1)\times D_+(x_0)$. On $U$ set $y_1=1=x_0$ and, by abuse of notation, write
\[
y_0=\frac{y_0}{y_1},\quad y_2=\frac{y_2}{y_1},\quad a=\frac{a}{y_1},\quad c=\frac{c}{y_1},\qquad
x_1=\frac{x_1}{x_0},\quad X_2=\frac{X_2}{x_0},\quad x_3=\frac{x_3}{x_0}.
\]
On $U$ the equation of $T$ has the form of a quadratic polynomial in the variable $c$:
\begin{equation}\label{eq:quad-c}
(y_2+a)\,c^2+(x_1^2-s y_0^2)\,c+\Bigl(y_0y_2x_1^2+a y_0(a-t y_0)X_2^2+F(y_0,1,y_2)x_3^2\Bigr)=0.
\end{equation}
Set
\[
A:=\bar K[y_0,y_2,a,x_1,X_2,x_3],\qquad L:=\Frac(A).
\]
Then \eqref{eq:quad-c} is a quadratic polynomial in $c$ over $L$, with discriminant
\[
\Delta:=(x_1^2-s y_0^2)^2-4(y_2+a)\Bigl(y_0y_2x_1^2+a y_0(a-t y_0)X_2^2+F(y_0,1,y_2)x_3^2\Bigr)\in A.
\]
As in Lemma \ref{lemma_generic_fiber_integral_birational}, it is enough to show that $\Delta$ is not a square in $L$. Consider the $\bar K$-algebra homomorphism $\varphi:A\to \bar K[x_1]$ given by
\[
y_0\mapsto 1,\qquad y_2\mapsto 1,\qquad a\mapsto 0,\qquad X_2\mapsto 0,\qquad x_3\mapsto 0.
\]
Then
\[
\varphi(\Delta)=(x_1^2-s)^2-4x_1^2=x_1^4-(2s+4)x_1^2+s^2\in \bar K[x_1].
\]
We claim $\varphi(\Delta)$ is not a square in $\bar K(x_1)$. Indeed, if $\varphi(\Delta)=q(x_1)^2$ with
$q(x_1)\in \bar K(x_1)$, then clearing denominators shows it is a square in $\bar K[x_1]$.
Write $q(x_1)=x_1^2+u x_1+v$ with $u,v\in \bar K$. Since $\Char(\bar K)\neq 2$ and $\varphi(\Delta)$ has no $x_1^3$ term,
we must have $u=0$, hence $q(x_1)^2=(x_1^2+v)^2=x_1^4+2v x_1^2+v^2$.
Comparing coefficients with $x_1^4-(2s+4)x_1^2+s^2$ yields $2v=-(2s+4)$, i.e. $v=-(s+2)$, hence $v^2=(s+2)^2\neq s^2$, a contradiction. Therefore $\varphi(\Delta)\notin \bar K(x_1)^{\times 2}$, so $\Delta\notin L^{\times 2}$. Thus the quadratic \eqref{eq:quad-c} is irreducible over $L$, hence irreducible in $A[c]$. 
\end{proof}
\begin{definition}
    Let 
    \begin{equation*}
        W_\cal X:=\{y_0x_0y_1x_1y_2x_2=0\}\subset \cal X. 
    \end{equation*} 
      Denote by $W_Y$ the special fiber and by $W_X$ (resp. $W_{\bar{X}}$) the generic (resp. geometric generic) fiber of $W_\cal X$. The complements are denoted as follows \begin{equation*}
        Y^\circ:=Y\setminus W_Y,\,\,Y_i^\circ:=Y_i\setminus W_Y,\,\,V^\circ:=V\setminus W_Y,\,\,X^\circ:=X_K\setminus W_X.
    \end{equation*}
\end{definition}
\begin{lemma}\label{Lemma_semistability}
    $\cal X ^\circ=\cal X-W_{\cal X}$ is strictly semistable.
\end{lemma}
\begin{proof}
On $\cal X^\circ$ we have $y_1\neq 0$ and $x_0\neq 0$, so $\cal X^\circ$ is contained in the affine chart
\[
U:=D_+(y_1)\times D_+(x_0).
\]
On $U$ set $y_1=1=x_0$ and, by abuse of notation, write the remaining affine coordinates as
\[
y_0=\frac{y_0}{y_1},\quad y_2=\frac{y_2}{y_1},\quad a=\frac{a}{y_1},\quad
x_1=\frac{x_1}{x_0},\quad X_2=\frac{X_2}{x_0},\quad x_3=\frac{x_3}{x_0}.
\]
Let $G_U$ be the dehomogenization of $G$ on $U$. Then $\cal X\cap U$ is cut out in
$\Spec k[y_0,y_2,a,c,d,x_1,X_2,x_3,s]$ by the two equations
\[
f:=cd+s=0,\qquad h:=G_U+x_1^2c+y_0^2d=0.
\]
We first check that $\cal X^\circ$ is regular. The Jacobian matrix of $(h,f)$ has the $2\times 2$ minor in the
columns $(c,s)$ equal to
\[
\det\begin{pmatrix}
\frac{\partial h}{\partial c} & \frac{\partial h}{\partial s}\\[2pt]
\frac{\partial f}{\partial c} & \frac{\partial f}{\partial s}
\end{pmatrix}
=
\det\begin{pmatrix}
x_1^2 & 0\\
d & 1
\end{pmatrix}
=x_1^2.
\]
Since $x_1\neq 0$ on $\cal X^\circ$ (because $x_1=0$ lies in $W_{\cal X}$), this determinant is invertible at every
point of $\cal X^\circ\cap U$. Hence the Jacobian has rank $2$ everywhere on $\cal X^\circ$, so $\cal X^\circ$ is regular.  Both $Y_0^\circ$ and $Y_1^\circ$ are Cartier divisors on the regular scheme $\cal X^\circ$. We check that $Y_0^\circ$ and $Y_1^\circ$ are smooth. On $Y_0^\circ$ we have $d=0$ and the remaining defining
equation is $h|_{d=0}=G_U+x_1^2c=0$, with
\[
\frac{\partial}{\partial c}(h|_{d=0})=x_1^2\neq 0
\]
on $\cal X^\circ$. Thus $Y_0^\circ$ is smooth. Similarly, $Y_1^\circ$ is smooth. Finally, $V$ is defined on the chart $U$ by the equation $G_U=0$ inside $\Spec k[y_0,y_2,a,x_1,X_2,x_3]$, and
\[
\frac{\partial G_U}{\partial x_1}=2(y_0y_2)x_1.
\]
By construction $\cal X^\circ\cap V$ avoids $\{y_0y_2x_1=0\}$, so $\frac{\partial G_U}{\partial x_1}\neq 0$ everywhere on $V^\circ$. Thus $V^\circ$ is smooth. Therefore $Y^\circ=Y_0^\circ\cup Y_1^\circ$ is a strict normal crossings divisor on the regular scheme $\cal X^\circ$,
and $\cal X^\circ\to \Spec(R)$ is strictly semistable.

\end{proof}

\begin{lemma}\label{lemma_chow_P4_P3}
    $\CH_1(Y_i^\circ\times_k L,\Lambda)=0$ for $i\in \{0,1\}$ and every field extension $L/k$.
\end{lemma}
\begin{proof}
   The claim follows from Lemma \ref{Lemma_simple_CH_1} as $\{x_0=0\}\cup \{y_0=0\}\subset W_X$.    
\end{proof}

\begin{proof}[Proof of Theorem \ref{theorem_3_2_in_4_3}]
The proof is identical to the proof of Theorem \ref{theorem_raise_deg_dim}.
     
\end{proof}

%%%%%%%%%%%%%

\section{Applications}
The aim of this section is to prove the following applications of the previous sections.
\Maintheorem*

\begin{proof}[Proof of Theorem \ref{Theorem_bidegree_bounds}]
We argue by cases.
\begin{caselist}
  \item By Proposition \ref{proposition_rel_tor_ord_HPT} we can apply Theorem \ref{theorem_raise_deg_dim} iteratively starting from the  $(2,2)$ hypersurface in $\bb P^3_k\times \bb P^2_k$ of Example \ref{example_HPT}, i.e. we may start with $n=4,r=1,d=2$ and raise the dimension and degree in the first component. 
  It then follows that for every $n-1\geq 3$ and 
    \begin{equation*}
        d\geq (n-1)-1=n-2
    \end{equation*}
    a very general bidegree $(d,2)$ hypersurface in $\bb P^{n-1}_k\times \bb P^2_k$ does not admit a decomposition of the diagonal.  
  \item For $n=5$ the claim follows from Proposition \ref{proposition_rel_tor_ord_HPT} applied to Example \ref{example_HPT}. For $n\geq 6$ we can apply Theorem \ref{theorem_raise_deg_dim} iteratively starting from the  $(3,2)$ hypersurface in $\bb P^4_k\times \bb P^3_k$ constructed in Theorem \ref{theorem_3_2_in_4_3}. It then follows that for 
    \begin{equation*}
        d\geq (n-2)-1=n-3
    \end{equation*}
    a very general bidegree $(d,2)$ hypersurface in $\bb P^{n-2}_k\times \bb P^3_k$ does not admit a decomposition of the diagonal.  
  \item 
    For $n\geq 5$ we can apply Theorem \ref{theorem_raise_deg_dim} iteratively starting from the $(2,3)$ hypersurface in $\bb P^3_k\times \bb P^3_k$ obtained by applying Theorem \ref{theorem_raise_deg_dim} to the bidegree $(2,2)$ hypersurface in $\bb P^3_k\times \bb P^2_k$ of Example \ref{example_HPT}. It then follows that for 
    \begin{equation*}
        d\geq (n-2)-1=n-3
    \end{equation*}
    a very general bidegree $(d,3)$ hypersurface in $\bb P^{n-2}_k\times \bb P^3_k$ does not admit a decomposition of the diagonal.  
  \qedhere
\end{caselist}
\end{proof}

\begin{proof}[Proof of Corollary \ref{Cor:reduction_bideg_to_cubic}]
    Let $t$ be a transcendental element over $k$ and denote $K=k(t)$. Let $c,M,L\in k[x_0,...,x_{r}]$ be very general hypersurfaces of the following degrees
    \begin{equation*}
        \deg(c)=3,\,\,\deg(M)=\deg(L)=2.
    \end{equation*}
    It follows from the proof of \cite[Theorem 1.4]{FiammengoLuedersCubics} (with $f=x_0^2$) and \cite[Corollary 6.11]{FiammengoLuedersCubics} that the quartic hypersurface
     \begin{equation*}
   Q:= \{x_{r+1}c+x_{r+1}^2M-tx_0^2L=0\}\subset \bb P^{r+1}_{\bar K}.
\end{equation*}
satisfies $\Tor (Q, \{x_0x_1x_{r+1}=0\})>1$. Consider the bidegree $(2,3)$ hypersurface
\[
  X:=\{y_0y_1c+y_1^2x_0M-ty_0^2x_0L=0\}\subset \Proj{k[y_0,y_1]}\times \Proj {k[x_0,..,x_{r}]}
     \]
     As $X$ and $Q$ agree on $D_+(x_0)\times D_+(y_0)$, respectively $D_+(x_0)$, they are birational and it follows from the above that $\Tor (X, \{y_0y_1x_0x_1=0\})>1$. We conclude by Lemma \ref{Lemma:very_general_decomposition} and an iterative application of Theorem \ref{theorem_raise_deg_dim}.
\end{proof}
\begin{remark}
    By \cite{EFS2025} a very general cubic threefold does not admit a decomposition of the diagonal, this together with Corollary \ref{Cor:reduction_bideg_to_cubic} give an alternative proof of Proposition \ref{prop:proposition_(2,3)_in_P^1timesP^4}.
\end{remark}
\begin{proof}[Proof of Corollary \ref{cor_2_4_in_1-2_6-5}]
    Consider the hypersurface 
    \begin{align*}
S_{2,4}:=\Bigl\{&
S_1^2W_0W_5^2\,(W_6+W_4)
+S_0S_1\,W_0W_6W_1^2
+S_0S_1\,W_4^2W_2^2
-tS_0^2\,W_0W_4W_2^2\\
&+\bigl((S_0-S_1)^2W_0^2+S_1^2W_6^2-2S_1(S_0+S_1)W_0W_6\bigr)W_3^2\\
&+S_1^2W_0W_1^2W_5
-sS_0^2W_0^3W_5=0\Bigr\}\subset\; \bb P^1_{[S_0:S_1]}\times \bb P^6_{[W_0:\dots:W_6]}.
\end{align*}
 We see that the dense open subset $D_+(S_1)\times D_+(W_0)\cap S_{2,4}$ is isomorphic to the open $D_+(y_1)\times D_+(x_0)\cap T$ for $T$ as in Lemma  \ref{lemma_substitute_coordinate_change}. More precisely under the isomorphism $\phi$ given by \[
\frac{y_0}{y_1}\mapsto \frac{S_0}{S_1},\quad
\frac{y_2}{y_1}\mapsto \frac{W_6}{W_0},\quad
\frac{a}{y_1}\mapsto \frac{W_4}{W_0},\quad
\frac{c}{y_1}\mapsto \frac{W_5}{W_0},\quad
\frac{x_1}{x_0}\mapsto \frac{W_1}{W_0},\quad
\frac{X_2}{x_0}\mapsto \frac{W_2}{W_0},\quad
\frac{x_3}{x_0}\mapsto \frac{W_3}{W_0}.
\]
we have \begin{align*}
&\left(\frac{c}{y_1}\right)^2\left(\frac{y_2}{y_1}+\frac{a}{y_1}\right)
+\frac{y_0}{y_1}\frac{y_2}{y_1}\left(\frac{x_1}{x_0}\right)^2
+\frac{a}{y_1}\frac{y_0}{y_1}\left(\frac{a}{y_1}-t\frac{y_0}{y_1}\right)\left(\frac{X_2}{x_0}\right)^2 \\
&\hspace{2.2cm}
+F\!\left(\frac{y_0}{y_1},1,\frac{y_2}{y_1}\right)\left(\frac{x_3}{x_0}\right)^2
+\left(\frac{x_1}{x_0}\right)^2\frac{c}{y_1}
-s\left(\frac{y_0}{y_1}\right)^2\frac{c}{y_1} \\
&\longmapsto\left(\frac{W_5}{W_0}\right)^2\left(\frac{W_6}{W_0}+\frac{W_4}{W_0}\right)
+\frac{S_0}{S_1}\frac{W_6}{W_0}\left(\frac{W_1}{W_0}\right)^2
+\frac{W_4}{W_0}\frac{S_0}{S_1}\left(\frac{W_4}{W_0}-t\frac{S_0}{S_1}\right)\left(\frac{W_2}{W_0}\right)^2 \\
&\hspace{2.2cm}
+F\!\left(\frac{S_0}{S_1},1,\frac{W_6}{W_0}\right)\left(\frac{W_3}{W_0}\right)^2
+\left(\frac{W_1}{W_0}\right)^2\frac{W_5}{W_0}
-s\left(\frac{S_0}{S_1}\right)^2\frac{W_5}{W_0}.
\end{align*}

 It follows from the proof of Theorem \ref{theorem_3_2_in_4_3} that for $\Lambda:=\bb Z/2$ we have $\Tor^\Lambda(S_{2,4},\{S_0W_0S_1W_1=0\})=2$. By setting $W=\{S_0W_0S_1W_1=0\}$ in Lemma \ref{Lemma:very_general_decomposition} we can conclude the proof.
\end{proof}

\bibliographystyle{acm}
\bibliography{Bibliografie}

@article{Springer,
 author = {Springer, Tonny Albert},
 title = {Sur les formes quadratiques d'indice z{\'e}ro},
 fjournal = {Comptes Rendus Hebdomadaires des S{\'e}ances de l'Acad{\'e}mie des Sciences, Paris},
 journal = {C. R. Acad. Sci., Paris},
 issn = {0001-4036},
 volume = {234},
 pages = {1517--1519},
 year = {1952},
 language = {French},
 zbMATH = {3072083},
 Zbl = {0046.24303}
}

@book{SerreGalCoh,
 author = {Serre, Jean-Pierre},
 title = {Galois cohomology. {Transl}. from the {French} by {Patrick} {Ion}.},
 edition = {2nd printing},
 fseries = {Springer Monographs in Mathematics},
 series = {Springer Monogr. Math.},
 issn = {1439-7382},
 isbn = {3-540-42192-0},
 year = {2002},
 publisher = {Berlin: Springer},
 language = {English},
 zbMATH = {1673451},
 Zbl = {1004.12003}
}

@article{Pirutka2025,
  author  = {Pirutka, Alena},
  title   = {Cubic Surface Bundles and the Brauer Group},
  journal = {International Mathematics Research Notices},
  year    = {2025},
  volume  = {2025},
  number  = {6},
  pages   = {rnaf063},
  doi     = {10.1093/imrn/rnaf063},
  url     = {https://academic.oup.com/imrn/article/2025/6/rnaf063/8081710}
}

@article{HassetKreshTschinkel,
 author = {Hassett, Brendan and Kresch, Andrew and Tschinkel, Yuri},
 title = {Stable rationality and conic bundles},
 fjournal = {Mathematische Annalen},
 journal = {Math. Ann.},
 issn = {0025-5831},
 volume = {365},
 number = {3-4},
 pages = {1201--1217},
 year = {2016},
 language = {English},
 doi = {10.1007/s00208-015-1292-y},
 url = {www.zora.uzh.ch/id/eprint/128400/1/srat.pdf},
 zbMATH = {6618530},
 Zbl = {1353.14019}
}

@misc{FiammengoLuedersCubics,
 author = {Fiammengo, Elia and L{\"u}ders, Morten},
 title = {On the diagonal of quartic hypersurfaces and $(2,3)$-complete intersection $n$-folds},
 year = {2025},
 howpublished = {Preprint, {arXiv}:2510.07111 [math.{AG}] (2025)},
 url = {https://arxiv.org/abs/2510.07111},
 arXiv = {arXiv:2510.07111}
}

@mastersthesis{Moethesis2020,
  author  = {Moe, Simen Westbye},
  title   = {Degenerations and Stable Rationality},
  school  = {University of Oslo},
  type    = {Master's thesis},
  address = {Oslo, Norway},
  year    = {2020}
}

@article{Merkurjev2008,
 author = {Merkurjev, Alexander},
 title = {Unramified elements in cycle modules},
 fjournal = {Journal of the London Mathematical Society. Second Series},
 journal = {J. Lond. Math. Soc., II. Ser.},
 issn = {0024-6107},
 volume = {78},
 number = {1},
 pages = {51--64},
 year = {2008},
 language = {English},
 doi = {10.1112/jlms/jdn011},
 zbMATH = {5309687},
 Zbl = {1155.14017}
}

@incollection{CTsurvey95,
 author = {Colliot-Th{\'e}l{\`e}ne, J.-L.},
 title = {Birational invariants, purity and the {Gersten} conjecture},
 booktitle = {\(K\)-theory and algebraic geometry: connections with quadratic forms and division algebras. Summer Research Institute on quadratic forms and division algebras, July 6-24, 1992, University of California, Santa Barbara, USA},
 isbn = {0-8218-0339-5},
 pages = {1--64},
 year = {1995},
 publisher = {Providence, RI: American Mathematical Society},
 language = {English},
 zbMATH = {753827},
 Zbl = {0834.14009}
}

@misc{LangeZhang,
 author = {Lange, Jan and Zhang, Guoyun},
 title = {Torsion order and irrationality of complete intersections},
 year = {2025},
 howpublished = {Preprint, {arXiv}:2510.24258 [math.{AG}] (2025)},
 url = {https://arxiv.org/abs/2510.24258},
 arXiv = {arXiv:2510.24258}
}

@article{CTPirutka:cyclic_covers,
 author = {Colliot-Th{\'e}l{\`e}ne, Jean-Louis and Pirutka, Alena V.},
 title = {Cyclic covers that are not stably rational},
 fjournal = {Izvestiya: Mathematics},
 journal = {Izv. Math.},
 issn = {1064-5632},
 volume = {80},
 number = {4},
 pages = {665--677},
 year = {2016},
 language = {English},
 doi = {10.1070/IM8429},
 zbMATH = {6640626},
 Zbl = {1375.14053}
}

@article{VoisinInventiones,
 author = {Voisin, Claire},
 title = {Unirational threefolds with no universal codimension {{\(2\)}} cycle},
 fjournal = {Inventiones Mathematicae},
 journal = {Invent. Math.},
 issn = {0020-9910},
 volume = {201},
 number = {1},
 pages = {207--237},
 year = {2015},
 language = {English},
 doi = {10.1007/s00222-014-0551-y},
 zbMATH = {6468725},
 Zbl = {1327.14223}
}

@article{Kuechle1995,
 author = {K{\"u}chle, Oliver},
 title = {On {Fano} 4-folds of index 1 and homogeneous vector bundles over {Grassmannians}},
 fjournal = {Mathematische Zeitschrift},
 journal = {Math. Z.},
 issn = {0025-5874},
 volume = {218},
 number = {4},
 pages = {563--575},
 year = {1995},
 language = {English},
 doi = {10.1007/BF02571923},
 url = {https://eudml.org/doc/174748},
 zbMATH = {784089},
 Zbl = {0826.14024}
}

@article{KuznetsovKuechle,
 author = {Kuznetsov, A. G.},
 title = {On {K{\"u}chle} varieties with {Picard} number greater than 1},
 fjournal = {Izvestiya: Mathematics},
 journal = {Izv. Math.},
 issn = {1064-5632},
 volume = {79},
 number = {4},
 pages = {698--709},
 year = {2015},
 language = {English},
 doi = {10.1070/IM2015v079n04ABEH002758},
 zbMATH = {6503848},
 Zbl = {1342.14087}
}

@book{Lam2005quadraticForms,
 author = {Lam, T. Y.},
 title = {Introduction to quadratic forms over fields},
 fseries = {Graduate Studies in Mathematics},
 series = {Grad. Stud. Math.},
 issn = {1065-7339},
 volume = {67},
 isbn = {0-8218-1095-2},
 year = {2005},
 publisher = {Providence, RI: American Mathematical Society (AMS)},
 language = {English},
 zbMATH = {2132158},
 Zbl = {1068.11023}
}

@misc{Voisin2025delPezzo,
 author = {Voisin, Claire},
 title = {Rank 2 vector bundles and degrees of points of del {Pezzo} surfaces},
 year = {2025},
 howpublished = {Preprint, {arXiv}:2509.17996 [math.{AG}] (2025)},
 url = {https://arxiv.org/abs/2509.17996},
 arXiv = {arXiv:2509.17996}
}

@misc{Voisi2025Fanos,
 author = {Voisin, Claire},
 title = {Unboundedness of zero-cycles on higher dimensional {Fano} manifolds},
 year = {2025},
 howpublished = {Preprint, {arXiv}:2511.23080 [math.{AG}] (2025)},
 url = {https://arxiv.org/abs/2511.23080},
 arXiv = {arXiv:2511.23080}
}

@article{AuelBigazziet.al.2021,
 author = {Auel, Asher and Bigazzi, Alessandro and B{\"o}hning, Christian and Graf von Bothmer, Hans-Christian},
 title = {Unramified {Brauer} groups of conic bundle threefolds in characteristic two},
 fjournal = {American Journal of Mathematics},
 journal = {Am. J. Math.},
 issn = {0002-9327},
 volume = {143},
 number = {5},
 pages = {1601--1631},
 year = {2021},
 language = {English},
 doi = {10.1353/ajm.2021.0040},
 url = {wrap.warwick.ac.uk/144158/1/WRAP-Unramified-Brauer-conic-bundle-threefolds-characteristic-two-Bigazzi-2020.pdf},
 zbMATH = {7441099},
 Zbl = {1484.14039}
}

@article{HassetPirutkaTschinkel2019,
 author = {Hassett, Brendan and Pirutka, Alena and Tschinkel, Yuri},
 title = {A very general quartic double fourfold is not stably rational},
 fjournal = {Algebraic Geometry},
 journal = {Algebr. Geom.},
 issn = {2313-1691},
 volume = {6},
 number = {1},
 pages = {64--75},
 year = {2019},
 language = {English},
 doi = {10.14231/AG-2019-004},
 zbMATH = {7020392},
 Zbl = {1535.14038}
}

@article{HassetTschinkel2019,
 author = {Hassett, Brendan and Tschinkel, Yuri},
 title = {On stable rationality of {Fano} threefolds and del {Pezzo} fibrations},
 fjournal = {Journal f{\"u}r die Reine und Angewandte Mathematik},
 journal = {J. Reine Angew. Math.},
 issn = {0075-4102},
 volume = {751},
 pages = {275--287},
 year = {2019},
 language = {English},
 doi = {10.1515/crelle-2016-0058},
 zbMATH = {7062937},
 Zbl = {1503.14014}
}

@article{AhmadinezhadOkada2018,
 author = {Ahmadinezhad, Hamid and Okada, Takuzo},
 title = {Stable rationality of higher dimensional conic bundles},
 fjournal = {{\'E}pijournal de G{\'e}om{\'e}trie Alg{\'e}brique. EPIGA},
 journal = {{\'E}pijournal de G{\'e}om. Alg{\'e}br., EPIGA},
 issn = {2491-6765},
 volume = {2},
 pages = {13},
 note = {Id/No 5},
 year = {2018},
 language = {English},
 url = {epiga.episciences.org/4575},
 zbMATH = {7003746},
 Zbl = {1422.14022}
}

@article{AuelBohningPirutka2018,
 author = {Auel, Asher and B{\"o}hning, Christian and Pirutka, Alena},
 title = {Stable rationality of quadric and cubic surface bundle fourfolds},
 fjournal = {European Journal of Mathematics},
 journal = {Eur. J. Math.},
 issn = {2199-675X},
 volume = {4},
 number = {3},
 pages = {732--760},
 year = {2018},
 language = {English},
 doi = {10.1007/s40879-018-0233-1},
 zbMATH = {7073885},
 Zbl = {1444.14029}
}

@incollection{Kollar2000,
 author = {Koll{\'a}r, J{\'a}nos},
 title = {Nonrational Covers of $CP^m\times CP^n$},
 booktitle = {Explicit birational geometry of 3-folds},
 isbn = {0-521-63641-8},
 pages = {51--71},
 year = {2000},
 publisher = {Cambridge: Cambridge University Press},
 language = {English},
 zbMATH = {1513898},
 Zbl = {0961.14033}
}

@article{OkadaKrylov2020,
 author = {Krylov, Igor and Okada, Takuzo},
 title = {Stable rationality of del {Pezzo} fibrations of low degree over projective spaces},
 fjournal = {IMRN. International Mathematics Research Notices},
 journal = {Int. Math. Res. Not.},
 issn = {1073-7928},
 volume = {2020},
 number = {23},
 pages = {9075--9119},
 year = {2020},
 language = {English},
 doi = {10.1093/imrn/rny252},
 url = {hdl.handle.net/21.11116/0000-0006-9DE7-5},
 zbMATH = {7323411},
 Zbl = {1457.14082}
}

@incollection{Schreiedersurvey,
 author = {Schreieder, Stefan},
 title = {Unramified cohomology, algebraic cycles and rationality},
 booktitle = {Rationality of varieties. Proceedings of the conference, Island of Schiermonnikoog, The Netherlands, spring 2019},
 isbn = {978-3-030-75420-4; 978-3-030-75423-5; 978-3-030-75421-1},
 pages = {345--388},
 year = {2021},
 publisher = {Cham: Birkh{\"a}user},
 language = {English},
 doi = {10.1007/978-3-030-75421-1_13},
 zbMATH = {7610317},
 Zbl = {1497.14010}
}

@article{NicaiseShinder2019,
 author = {Nicaise, Johannes and Shinder, Evgeny},
 title = {The motivic nearby fiber and degeneration of stable rationality},
 fjournal = {Inventiones Mathematicae},
 journal = {Invent. Math.},
 issn = {0020-9910},
 volume = {217},
 number = {2},
 pages = {377--413},
 year = {2019},
 language = {English},
 doi = {10.1007/s00222-019-00869-2},
 zbMATH = {7088056},
 Zbl = {1455.14029}
}

@article{Schreieder2018_quadricbundles,
 author = {Schreieder, Stefan},
 title = {Quadric surface bundles over surfaces and stable rationality},
 fjournal = {Algebra \& Number Theory},
 journal = {Algebra Number Theory},
 issn = {1937-0652},
 volume = {12},
 number = {2},
 pages = {479--490},
 year = {2018},
 language = {English},
 doi = {10.2140/ant.2018.12.479},
 zbMATH = {6880896},
 Zbl = {1397.14026}
}

@article{Mboro2019,
 author = {Mboro, Ren{\'e}},
 title = {Remarks on approximate decompositions of the diagonal},
 fjournal = {Communications in Algebra},
 journal = {Commun. Algebra},
 issn = {0092-7872},
 volume = {47},
 number = {7},
 pages = {2995--3002},
 year = {2019},
 language = {English},
 doi = {10.1080/00927872.2018.1546864},
 zbMATH = {7068127},
 Zbl = {1444.14023}
}

@article{Voisin_Cubics2017,
 author = {Voisin, Claire},
 title = {On the universal {{\(\mathrm{CH}_0\)}} group of cubic hypersurfaces},
 fjournal = {Journal of the European Mathematical Society (JEMS)},
 journal = {J. Eur. Math. Soc. (JEMS)},
 issn = {1435-9855},
 volume = {19},
 number = {6},
 pages = {1619--1653},
 year = {2017},
 language = {English},
 doi = {10.4171/JEMS/702},
 zbMATH = {6725172},
 Zbl = {1366.14009}
}

@article{ChatzistamatiouLevine,
 author = {Chatzistamatiou, Andre and Levine, Marc},
 title = {Torsion orders of complete intersections},
 fjournal = {Algebra \& Number Theory},
 journal = {Algebra Number Theory},
 issn = {1937-0652},
 volume = {11},
 number = {8},
 pages = {1779--1835},
 year = {2017},
 language = {English},
 doi = {10.2140/ant.2017.11.1779},
 zbMATH = {6806362},
 Zbl = {1453.14024}
}

@article{Vial2013,
 author = {Vial, Charles},
 title = {Algebraic cycles and fibrations},
 fjournal = {Documenta Mathematica},
 journal = {Doc. Math.},
 issn = {1431-0635},
 volume = {18},
 pages = {1521--1553},
 year = {2013},
 language = {English},
 doi = {10.4171/dm/435},
 zbMATH = {6269756},
 Zbl = {1349.14027}
}

@article{Hartl2001,
 author = {Hartl, Urs T.},
 title = {Semi-stability and base change},
 fjournal = {Archiv der Mathematik},
 journal = {Arch. Math.},
 issn = {0003-889X},
 volume = {77},
 number = {3},
 pages = {215--221},
 year = {2001},
 language = {English},
 doi = {10.1007/PL00000484},
 zbMATH = {1690994},
 Zbl = {1071.14504}
}

@article{Totaro2016,
 author = {Totaro, Burt},
 title = {Hypersurfaces that are not stably rational},
 fjournal = {Journal of the American Mathematical Society},
 journal = {J. Am. Math. Soc.},
 issn = {0894-0347},
 volume = {29},
 number = {3},
 pages = {883--891},
 year = {2016},
 language = {English},
 doi = {10.1090/jams/840},
 zbMATH = {6572969},
 Zbl = {1376.14017}
}

@misc{EFS2025,
 author = {Engel, Philip Milton and Fortman, Olivier de Gaay and Schreieder, Stefan},
 title = {Matroids and the integral {Hodge} conjecture for abelian varieties},
 year = {2025},
 howpublished = {Preprint, {arXiv}:2507.15704 [math.{AG}] (2025)},
 url = {https://arxiv.org/abs/2507.15704},
 arXiv = {arXiv:2507.15704}
}

@article{SchreiederAMS,
 author = {Schreieder, Stefan},
 title = {Stably irrational hypersurfaces of small slopes},
 fjournal = {Journal of the American Mathematical Society},
 journal = {J. Am. Math. Soc.},
 issn = {0894-0347},
 volume = {32},
 number = {4},
 pages = {1171--1199},
 year = {2019},
 language = {English},
 doi = {10.1090/jams/928},
 zbMATH = {7121119},
 Zbl = {1442.14138}
}

@article{PavicSchreieder2023,
 author = {Pavic, Nebojsa and Schreieder, Stefan},
 title = {The diagonal of quartic fivefolds},
 fjournal = {Algebraic Geometry},
 journal = {Algebr. Geom.},
 issn = {2313-1691},
 volume = {10},
 number = {6},
 pages = {754--778},
 year = {2023},
 language = {English},
 doi = {10.14231/AG-2023-027},
 zbMATH = {7774538},
 Zbl = {1540.14089}
}

@article{Schreieder2019quadricbundles,
 author = {Schreieder, Stefan},
 title = {On the rationality problem for quadric bundles},
 fjournal = {Duke Mathematical Journal},
 journal = {Duke Math. J.},
 issn = {0012-7094},
 volume = {168},
 number = {2},
 pages = {187--223},
 year = {2019},
 language = {English},
 doi = {10.1215/00127094-2018-0041},
 zbMATH = {7036862},
 Zbl = {1409.14030}
}

@article{HassetPirutkaTschinkel2018,
 author = {Hassett, Brendan and Pirutka, Alena and Tschinkel, Yuri},
 title = {Stable rationality of quadric surface bundles over surfaces},
 fjournal = {Acta Mathematica},
 journal = {Acta Math.},
 issn = {0001-5962},
 volume = {220},
 number = {2},
 pages = {341--365},
 year = {2018},
 language = {English},
 doi = {10.4310/ACTA.2018.v220.n2.a4},
 zbMATH = {6925267},
 Zbl = {1420.14115}
}

@misc{LangeSchreieder2024,
 author = {Lange, Jan and Schreieder, Stefan},
 title = {On the rationality problem for low degree hypersurfaces},
 year = {2024},
 howpublished = {Preprint, {arXiv}:2409.12834 [math.{AG}] (2024)},
 url = {https://arxiv.org/abs/2409.12834},
 arXiv = {arXiv:2409.12834}
}

@article{Skauli2023,
 author = {Skauli, Bj{\o}rn},
 title = {A {{\((2,3)\)}}-complete intersection fourfold with no decomposition of the diagonal},
 fjournal = {Manuscripta Mathematica},
 journal = {Manuscr. Math.},
 issn = {0025-2611},
 volume = {171},
 number = {3-4},
 pages = {473--486},
 year = {2023},
 language = {English},
 doi = {10.1007/s00229-022-01386-y},
 zbMATH = {7694751},
 Zbl = {1517.14010}
}

@article{NicaiseOttem2022,
 author = {Nicaise, Johannes and Ottem, John Christian},
 title = {Tropical degenerations and stable rationality},
 fjournal = {Duke Mathematical Journal},
 journal = {Duke Math. J.},
 issn = {0012-7094},
 volume = {171},
 number = {15},
 pages = {3023--3075},
 year = {2022},
 language = {English},
 doi = {10.1215/00127094-2022-0065},
 url = {hdl.handle.net/10852/97983},
 zbMATH = {7608257},
 Zbl = {1509.14121}
}

@misc{LangeSkauli2023,
 author = {Lange, Jan and Skauli, Bj{\o}rn},
 title = {The diagonal of (3,3) fivefolds},
 year = {2023},
 howpublished = {Preprint, {arXiv}:2303.00562 [math.{AG}] (2023)},
 url = {https://arxiv.org/abs/2303.00562},
 arXiv = {arXiv:2303.00562}
}

@misc{Moe2023,
 author = {Moe, Simen Westbye},
 title = {On {Stable} {Rationality} of {Polytopes}},
 year = {2023},
 howpublished = {Preprint, {arXiv}:2311.01144 [math.{AG}] (2023)},
 url = {https://arxiv.org/abs/2311.01144},
 arXiv = {arXiv:2311.01144}
}

@article {Schreieder2021,
    AUTHOR = {Schreieder, Stefan},
     TITLE = {Torsion orders of {F}ano hypersurfaces},
   JOURNAL = {Algebra Number Theory},
  FJOURNAL = {Algebra \& Number Theory},
    VOLUME = {15},
      YEAR = {2021},
    NUMBER = {1},
     PAGES = {241--270},
      ISSN = {1937-0652},
   MRCLASS = {14J70 (14C25 14E08 14J45 14M20)},
  MRNUMBER = {4226988},
       DOI = {10.2140/ant.2021.15.241},
       URL = {https://doi.org/10.2140/ant.2021.15.241},
}

@book {Kollar1996,
    AUTHOR = {Koll\'{a}r, J\'{a}nos},
     TITLE = {Rational curves on algebraic varieties},
    SERIES = {Ergebnisse der Mathematik und ihrer Grenzgebiete. 3. Folge. A
              Series of Modern Surveys in Mathematics [Results in
              Mathematics and Related Areas. 3rd Series. A Series of Modern
              Surveys in Mathematics]},
    VOLUME = {32},
 PUBLISHER = {Springer-Verlag, Berlin},
      YEAR = {1996},
     PAGES = {viii+320},
      ISBN = {3-540-60168-6},
   MRCLASS = {14-02 (14C05 14E05 14F17 14J45)},
  MRNUMBER = {1440180},
MRREVIEWER = {Yuri G. Prokhorov},
       DOI = {10.1007/978-3-662-03276-3},
       URL = {https://doi.org/10.1007/978-3-662-03276-3},
}

@book {Fulton1998,
    AUTHOR = {Fulton, William},
     TITLE = {Intersection theory},
    SERIES = {Ergebnisse der Mathematik und ihrer Grenzgebiete. 3. Folge. A
              Series of Modern Surveys in Mathematics [Results in
              Mathematics and Related Areas. 3rd Series. A Series of Modern
              Surveys in Mathematics]},
    VOLUME = {2},
   EDITION = {Second},
 PUBLISHER = {Springer-Verlag, Berlin},
      YEAR = {1998},
     PAGES = {xiv+470},
      ISBN = {3-540-62046-X; 0-387-98549-2},
   MRCLASS = {14C17 (14-02)},
  MRNUMBER = {1644323},
       DOI = {10.1007/978-1-4612-1700-8},
       URL = {http://dx.doi.org/10.1007/978-1-4612-1700-8},
}

@article {Ro96,
    AUTHOR = {Rost, Markus},
     TITLE = {Chow groups with coefficients},
   JOURNAL = {Doc. Math.},
  FJOURNAL = {Documenta Mathematica},
    VOLUME = {1},
      YEAR = {1996},
     PAGES = {No. 16, 319--393},
      ISSN = {1431-0635},
   MRCLASS = {14C15 (14C35 19D45)},
  MRNUMBER = {1418952},
MRREVIEWER = {Claudio Pedrini},
}

@misc{stacks-project,
    shorthand    = {Stacks},
    author       = {{The Stacks Project Authors}},
    title        = {\textit{Stacks Project}},
    howpublished = {\url{https://stacks.math.columbia.edu}},
    year         = {2018},
  }

\noindent
\parbox{0.5\linewidth}{
\noindent
Elia Fiammengo \\ 
Universität Heidelberg\\
Mathematisches Institut \\
Im Neuenheimer Feld 205 \\
69120 Heidelberg \\
Germany\\
{\tt 	elia.fiammengo@stud.uni-heidelberg.de}
}
\\
\newline

\noindent
\parbox{0.5\linewidth}{
\noindent
Morten L\"uders \\ 
Universität Heidelberg\\
Mathematisches Institut \\
Im Neuenheimer Feld 205 \\
69120 Heidelberg \\
Germany\\
{\tt mlueders@mathi.uni-heidelberg.de}
}
\end{document}